%% file: one-dimensional-cohomology.tex
\documentclass[a4paper,12pt%
]{amsart}
\pdfoutput=1
\usepackage{amsmath, amssymb, amsthm, amscd}
\usepackage{graphicx}
\usepackage{enumerate}
\usepackage[all,cmtip]{xy}

\newcounter{theorem}
\newtheorem{thm}[theorem]{Theorem}

\newtheorem{prop}[theorem]{Proposition}
\newtheorem{cor}[theorem]{Corollary}
\newtheorem{defn}[theorem]{Definition}

\theoremstyle{remark}
\newtheorem*{remark*}{Remark}
\newtheorem{remark}[theorem]{Remark}
\newtheorem{example}[theorem]{Example}

\numberwithin{equation}{section}
\numberwithin{theorem}{section}

\DeclareMathOperator\spa{span}
\DeclareMathOperator\diam{diam}

\newcommand{\alabel}{\label}
\newcommand{\defemph}{\emph}

\newcommand{\R}{\mathbb{R}}
\newcommand{\Q}{\mathbb{Q}}
\newcommand{\Z}{\mathbb{Z}}
\newcommand{\N}{\mathbb{N}}
\newcommand{\A}{\mathcal{A}}
\newcommand{\TS}{\Omega}
\newcommand{\SSS}{\Sigma}
\newcommand{\tilet}{t}
\newcommand{\wdw}{W}
\newcommand{\wdu}{U}

\newcommand{\lru}{r}
\newcommand{\lrv}{s}
\newcommand{\lrx}{x}
\newcommand{\lry}{y}
\newcommand{\lrz}{z}
\newcommand{\lrw}{t}
\newcommand{\sub}{\phi}

\newcommand{\shift}{\gamma}
\newcommand{\seqshift}{\sigma}
\newcommand{\inflation}{\lambda}
\newcommand{\boundary}{\partial}
\newcommand{\supp}{Supp}

\newcommand{\collar}{L}

\newcommand{\cellc}{c}

\newcommand{\cochainc}{c'}

\newcommand{\tile}{t}
\newcommand{\tiling}{T}

\newcommand{\vectorv}{v}
\newcommand{\basisb}{{\mathfrak B}}

\newcommand{\family}{F}
\newcommand{\seq}{s}
\newcommand{\seqr}{r}

\newcommand{\complexheight}{6cm}
\newcommand{\fullcomplexheight}{3.7cm}

\title[Cohomology of Mixed Substitution Tiling Spaces]{Cohomology of One-dimensional Mixed Substitution Tiling Spaces}

\author{Franz G\"{a}hler}
\address{Bielefeld University}
\thanks{Both authors were partly supported by the German Research Council (DFG), via CRC 701}

\author{Gregory R. Maloney}
\address{University of Massachusetts Boston}

\begin{document}

\date{March 5, 2012}

\begin{abstract}
We compute the Cech cohomology with integer coefficients of one-dimensional tiling spaces arising from not just one, but several different substitutions, all acting on the same set of tiles.   
These calculations involve the introduction of a universal version of the Anderson-Putnam complex.  
We show that, under a certain condition on the substitutions, the projective limit of this universal Anderson-Putnam complex is isomorphic to the tiling space, and we introduce a simplified universal Anderson-Putnam complex that can be used to compute Cech cohomology.  
We then use this simplified complex to place bounds on the rank of the first cohomology group of a one-dimensional substitution tiling space in terms of the number of tiles.
\end{abstract}

\maketitle

\input{intro3}

\input{full-complex3}

\input{rank-bound3}

\bibliographystyle{abbrv}
\bibliography{bib-cohomology}
\end{document}

%% file: intro3.tex
\section{Introduction and Definitions}\alabel{SEC:intro}

The purpose of this work is to investigate the cohomology groups of tilings obtained by mixing several different substitutions.  
Many of the results can be proven for tilings in arbitrary dimension, so the notation and definitions for general tiling spaces will be introduced in Sections \ref{SEC:tilings}, \ref{SEC:substitutions}, and \ref{SEC:mixed-substitutions}.  
Nevertheless, special attention will be given to the class of one-dimensional tilings, and in this context it is often easier to work with symbolic substitutions and subshifts; accordingly, the relevant notions for symbolic substitutions will be introduced in Section \ref{SEC:symbolic-substitutions}.  

This work follows the paper \cite{AP} very closely, and many of the definitions and notations relating to tilings are taken from that source.  
An excellent introduction to the theory of topology of tiling spaces can be found in \cite{S:book}.  
The definitions and notations relating to symbolic shift spaces are standard; see \cite{LM} for an introduction.  

\input{tilings}

\input{substitutions}

\input{mixed-substitutions}
\input{symbolic-substitutions}
\input{ap-complex}

%% file: tilings.tex
\subsection{Tilings}\alabel{SEC:tilings}

A \defemph{tile} is a subset of $\R^d$ that is homeomorphic to the closed unit ball.  
A \defemph{partial tiling} $T$ is a set of tiles, any two of which intersect only on their boundaries (let us denote the boundary of a set $S$ by $\boundary(S)$).  
The \defemph{support} of $T$, denoted $\supp(T)$, is the union of its tiles.  
A \defemph{tiling} is a partial tiling, the support of which is $\R^d$.  
When we need different tiles that look alike, let us associate a label with each tile; in such cases, let us consider a tile to be an ordered pair consisting of the set and the label.  
Given a partial tiling $\tiling$ and a vector $u\in \R^d$, define the translation of $\tiling$ by $u$ to be
\[
\tiling + u = \{ \tile + u : \tile \in \tiling\},
\]
where, for a tile $\tile$, $\tile + u = \{ x + u : x\in \tile\}$.

Any set of tilings of $\R^d$ can be equipped with a metric, in which two tilings are close if, up to a small translation, they agree on a large ball around the origin.  
There are several ways to define a metric in this way, all of which give rise to the same topology.  
Let us use the metric defined in \cite{AP}: for any two tilings $\tiling, \tiling'$ of $\R^d$,
\begin{equation*}
\begin{split}
D(\tiling,\tiling') & := \inf\left( \{ 1/\sqrt{2} \} \cup \{\epsilon : \tiling + u \textrm{ and } \tiling' + v \textrm{ agree on } B_{1/\epsilon}(0) \right. \\%
& \quad \left.  \textrm{ for some } \| u\|, \| v\| < \epsilon\}\right).
\end{split}
\end{equation*}

With respect to the topology arising from this metric, the action of $\R^d$ by translation is continuous.

%% file: substitutions.tex
\subsection{Substitutions}\alabel{SEC:substitutions}

Let $\{ p_1, \ldots , p_k\}$ be a finite set of tiles, called \defemph{prototiles}.  
Let $\tilde{\TS}$ denote the set of all partial tilings that contain only translates of these prototiles.  
A \defemph{substitution} $\sub$ is a map from $\{p_1, \ldots, p_k\}$ to $\tilde{\TS}$ for which there exists an \defemph{inflation constant} $\inflation > 1$ such that, for all $i \leq k$, the support of $\sub(p_i)$ is $\inflation p_i$.  
Then $\sub$ can be extended to a map $\sub : \tilde{\TS} \to \tilde{\TS}$ by 
\[
\sub(\tiling) = \bigcup_{p_i+u\in\tiling}(\sub(p_i) + \inflation u).
\]
Then the \defemph{tiling space} or \defemph{hull} $\TS_\sub$ is the set of all tilings $\tiling \in \tilde{\TS}$ such that, for any partial tiling $P\subseteq \tiling$ with bounded support, we have $P \subseteq \sub^n(p_i+u)$ for some prototile $p_i$ and some vector $u$.  
Note that $\sub(\TS_\sub) \subseteq \TS_\sub$.  

\begin{remark}\label{REM:conditions}
The substitution tiling spaces considered in \cite{AP} all satisfy the following three conditions.
\begin{enumerate}
\item  $\sub$ is one-to-one on $\TS_\sub$.  
This is required in order for $\sub|_{\TS_\sub}$ to have an inverse.  
By \cite{S:aperiodic}, $\sub$ is one-to-one on $\TS_\sub$ if and only if $\TS_\sub$ consists only of non-periodic tilings.  
\item  $\sub$ is primitive.  
This means that there exists some $n\geq 1$ such that, for any two prototiles $p_i, p_j$, some translate of $p_i$ appears in $\sub^n(p_j)$.  
\item  $\TS_\sub$ has \defemph{finite local complexity} (FLC).  
This means that, for each positive real number $R$, there are, up to translation, only finitely many partial tilings that are subsets of tilings in $\TS_\sub$ and that have supports with diameter less than $R$.
\end{enumerate}

Let us consider only substitutions that satisfy these three conditions, in addition to the following extra condition, which is a hypothesis of some of the theorems in \cite{AP}.  
\begin{enumerate}
\setcounter{enumi}{3}
\item  The prototiles of $\sub$ have a $CW$-structure with respect to which the tilings in $\TS_\sub$ are \defemph{edge-to-edge}, which means that, given any two subcells of tiles in the same tiling, their intersection is a union of subcells.
\end{enumerate}
\end{remark}


%% file: mixed-substitutions.tex
\subsection{Mixed Substitution Systems}\alabel{SEC:mixed-substitutions}

The goal in this section is to generalize the notion of a hull by allowing more than one substitution to be used.  
This motivates the following definition.  
\begin{defn}\alabel{DEF:multi-sub}
Let $\family = \{ \sub_1,\sub_2,\ldots , \sub_k\}$ be a finite set of substitutions all acting on the same set of prototiles $\{ p_1, p_2, \ldots , p_l\}$ in $\R^d$, and consider an infinite sequence $\seq = (\seq_1, \seq_2,\seq_3,\ldots ) \in \{1, 2, \ldots , k\}^\N$.  
As before, let $\tilde{\TS}$ denote the set of all partial tilings containing only translates of the prototiles.  
Then the \defemph{mixed substitution space} or \defemph{hull} of $\family$ and $\seq$ is denoted by $\TS_{\family,\seq}$ and consists of all tilings $\tiling$ in $\tilde{\TS}$ for which every $P \subseteq \tiling$ with bounded support is contained in $\sub_{\seq_1}\sub_{\seq_2}\cdots \sub_{\seq_n}(p_i+u)$ for some natural number $n$, some prototile $p_i$, and some translation vector $u$.  
\end{defn}
The inflation factors for the substitutions $\sub_1,\ldots,\sub_k$ might be different.  
In the notation of Definition \ref{DEF:multi-sub}, let us denote the inflation factor of $\sub_{\seq_i}$ by $\inflation_{\seq_i}$.  
\begin{remark}
It remains to be shown that $\TS_{\family,\seq}$ is non-empty; the proof of this fact appears below, for the special case in which $(\family,\seq)$ is primitive (see Definition \ref{DEF:mixed-primitive}).  
\end{remark}

Let us assume further that the substitutions in $\family$ also satisfy the following compatibility condition.  
\begin{defn}\alabel{DEF:compatible}
Let $P = \{ p_1, p_2, \ldots , p_l\}$ be a set of prototiles, each of which has a $CW$-structure, and let $\family = \{ \sub_1, \sub_2, \ldots , \sub_k\}$ be a finite family of substitutions on $P$.  
$\family$ is \defemph{compatible} if, for all $i\leq l$, for all $n\in\N$, and all $(\seqr_1,\seqr_2,\ldots,\seqr_n)\in\{1,2,\ldots,k\}^n$, the intersection of any tile $\tilet\in\sub_{\seqr_1}\sub_{\seqr_2}\cdots\sub_{\seqr_n}(p_i)$ with any other such tile is a union of subcells of $\tilet$.  
\end{defn}

\begin{remark}
If $\family$ is compatible in the sense of Definition \ref{DEF:compatible}, then, for each sequence $\seq$, the tilings in $\TS_{\family,\seq}$ will be edge-to-edge in the sense defined in Section \ref{SEC:substitutions}.  

The compatibility property is automatic for one-dimensional tiles.  
\end{remark}
As before, this hull can be equipped with the tiling space topology, with respect to which $\R^d$ acts continuously by translation.  
Then there is a natural extension of the definition of primitivity that is sufficient to guarantee that the $\R^d$ action be minimal.
\begin{defn}\alabel{DEF:mixed-primitive}
Let $\family = \{ \sub_1,\sub_2,\ldots , \sub_k\}$ be a finite set of substitutions all acting on the same set of prototiles $\{ p_1, p_2, \ldots , p_l\}$ in $\R^d$, and let $\seq = (\seq_1, \seq_2,\seq_3,\ldots ) \in \{1, 2, \ldots , k\}^\N$ be an infinite sequence.  
$(\family,\seq)$ is called \defemph{primitive} if there exists $n$ such that, for all $i,j\leq l$ and all $m\in\N$, $\sub_{\seq_{m+1}}\sub_{\seq_{m+2}}\cdots $ $\sub_{\seq_{m+n}}(p_i)$ contains a translate of $p_j$.  

Let us say that $\family$ is primitive if there exists $n$ such that, for all $i,j\leq l$ and all $(\seqr_1,\seqr_2,\ldots ,\seqr_n) \in \{ 1,2\ldots , k\}^n$, $\sub_{\seqr_1}\sub_{\seqr_2}\cdots \sub_{\seqr_n}(p_i)$ contains a translate of $p_j$.  
\end{defn}
\begin{remark}
If $(\family,\seq)$ is primitive, then $(\TS_{\family,\seq},\R^d)$ is a minimal dynamical system.  
If $\family$ is primitive, then $(\family,\seq)$ is primitive for every sequence $\seq$, so $(\TS_{\family,\seq},\R^d)$ is minimal for every hull $\TS_{\family,\seq}$, though not necessarily on the joint hull of several sequences $(\family, \seq)$.  

This generalizes the standard definition of primitivity because $\sub$ is primitive in the standard sense if and only if $\{ \sub\}$ is primitive in this sense.

If $\seq_1$ and $\seq_2$ are two sequences, then by minimality the two hulls $\TS_{\family,\seq_1}$ and $\TS_{\family,\seq_2}$ are either the same or disjoint.  
That they can be the same, can be seen with the two Fibonacci substitutions 
\[
\sub_1 : \begin{array}{ll} 
a & \mapsto  b\\
b & \mapsto ab
\end{array},\quad \begin{array}{ll}
a & \mapsto b\\
b & \mapsto ba
\end{array},
\]

which can be freely mixed, and always produce Fibonacci tilings.   

That they can be different can be seen in Example \ref{EX:long-example}, in which there appears a family $\family$ that is primitive in the sense of Definition \ref{DEF:mixed-primitive} (indeed, we may take $n = 4$), and for which some of the hulls $\TS_{\family,\seq_1}$ and $\TS_{\family,\seq_2}$ can be distinguished by the ranks of their cohomology groups.  
\end{remark}

If $\{ \sub_1,\sub_2,\ldots ,\sub_k\}$ is primitive, then each of $\sub_1, \sub_2,\ldots , \sub_k$ must be primitive by itself, but the converse of this is not true, as can be seen in the following example.
\begin{example}\alabel{EX:not-mixed-primitive}
Consider the two substitutions on $\A = \{ a, b\}$ given by
\[
\sub_1 : \left\{ \begin{array}{rcl}
a & \to & bb \\%
b & \to & aba %
\end{array}\right. ,%
\qquad %
\sub_2 : \left\{ \begin{array}{rcl}
a & \to & aab \\%
b & \to & aa %
\end{array}\right. .%
\]
Then each of $\sub_1$ and $\sub_2$ is primitive, but $\sub_2\sub_1$ is not primitive, and so $\{ \sub_1, \sub_2\}$ is not primitive either.  
\end{example}

\begin{remark}
If $(\family,\seq)$ is primitive, then $\TS_{\family,\seq}$ is non-empty.  
The proof of this is a modification of the standard argument that is used to show that $\TS_\sub$ is non-empty for a primitive substitution $\sub$.  

For each $i\in \N$, let $\inflation_{\seq_i}$ be the inflation factor of the substitution $\sub_{\seq_i}$.  
The primitivity condition implies that there is some $N_1 > 0$ such that, for some prototile $p_1$, $\sub_{\seq_1}\cdots \sub_{\seq_{N_1}}(p_1)$ contains a translate $p_1+v_1$ of $p_1$ in its interior.  
Likewise, there is some $N_2 > 0$ such that $\sub_{\seq_{N_1+1}}\cdots \sub_{\seq_{N_1+N_2}}(p_1)$ contains a translate $p_1+v_2$ of $p_1$ in its interior.  
Proceding in this fashion results in a sequence of patches
\[
P_1\subseteq P_2\subseteq P_3 \subseteq \cdots
\]
containing the origin.  
Here $P_i = \sub_{\seq_{1}}\cdots\sub_{\seq_{N_1+\cdots +N_i}}(p_1)-v_1-\inflation_{\seq_1} v_2 - \cdots - \inflation_{\seq_1} \cdots \inflation_{\seq_{i-1}}v_i$.  
Substracting the vector $v_1 + \inflation_{\seq_1} v_2 + \cdots + \inflation_{\seq_1}\cdots \inflation_{\seq_{i-1}} v_i$ guarantees that $P_i\subseteq P_{i+1}$.  
Furthermore, the sequence of patches expands to cover $\R^d$, and so defines a tiling of $\R^d$.  
From its construction, it is clear that this tiling is in $\TS_{\family,\seq}$.  
\end{remark}
\begin{remark}\alabel{REM:mixed-conditions}
Let us now give four conditions for a mixed substitution system $\TS_{\family,\seq}$ that are analogous to the conditions for ordinary substitution systems described in Remark \ref{REM:conditions}.  
Condition 1 involves the shift operator $\seqshift$ on one-sided sequences, which is defined by
\[
\seqshift (\seq_1,\seq_2, \ldots ) = (\seq_2,\seq_3,\ldots).
\]
\begin{enumerate}
\item  The map $\sub_{s_i}$ is a one-to-one map from $\TS_{\family,\seqshift^{i}\seq}$ to $\TS_{\family,\seqshift^{i-1}\seq}$.  
\item  $(\family,\seq)$ is primitive.  
\item  $\TS_{\family,\seq}$ has FLC.  
\item  $\family$ is compatible.  
\end{enumerate}
Let us assume henceforth that all of the mixed substitution tiling spaces described here satisfy these four conditions.  
\end{remark}

%% file: symbolic-substitutions.tex
\subsection{Symbolic Substitutions}\alabel{SEC:symbolic-substitutions}

There is a simple way of describing one-dimensional tiling systems in terms of purely symbolic information.  

Let $\A$ be a finite set of symbols, called an \defemph{alphabet}.  
Let $\A^n$ denote the set of all words of length $n$, the letters of which are elements of $\A$.  
Let $\A^*$ denote $\bigcup_{n\geq 1} \A^n$, the set of all words of any length, the letters of which are elements of $\A$.  
Let $|\wdw|$ denote the length of a word $\wdw$.  
Given a word $\wdw = \lrx_1 \lrx_2 \ldots \lrx_k$ and numbers $i \leq j \leq k$, let us denote by $\wdw_{[i,j]}$ the \defemph{subword} $\lrx_i \lrx_{i+1} \ldots \lrx_j$.  
If $i = j$, then let us write $\wdw_{[i]}$ instead of $\wdw_{[i,i]}$.  
If $\wdw = \lrx_1 \lrx_2 \ldots \lrx_k$ and $\wdu = \lry_1 \lry_2 \ldots \lry_l$ are words, then let $\wdw\wdu$ denote the \defemph{concatenation} of $\wdw$ and $\wdu$; that is, $\wdw\wdu = \lrx_1\lrx_2 \ldots \lrx_k \lry_1\lry_2\ldots \lry_l$.

Given a word $\wdw = \lrx_1 \lrx_2 \ldots \lrx_n \in \A^n$ and integers $1 \leq l \leq m \leq n$, let $\wdw_{[-m,-l]} = \lrx_{n-m+1}\lrx_{n-m+2}\ldots \lrx_{n-l+1}$.  
If $l = m$, then let us write $\wdw_{[-l]}$ instead of $\wdw_{[-l,-l]}$.  
Given two words $\wdw, \wdu \in \A^*$, let $\delta_{\wdw,\wdu}$ denote the \defemph{Kronecker delta function} of $\wdw$ and $\wdu$; that is,
\[
\delta_{\wdw,\wdu} = \left\{ \begin{array}{ll}
1 & \textrm{if } \wdw = \wdu \\
0 & \textrm{otherwise }
\end{array}\right. .
\]

A \defemph{symbolic substitution} on $\A$ is a map $\sub : \A \to \A^*$.  
Any substitution $\sub$ extends naturally to a map---let us also denote this by $\sub$---from $\A^*$ to $\A^*$, defined by setting $\sub(\lrx_1 \lrx_2 \ldots \lrx_k) = \sub(\lrx_1)\sub(\lrx_2)\ldots \sub(\lrx_k)$.

There is a notion of primitivity for symbolic substitutions that is exactly analogous to the notion of primitivity for substitution tiling spaces.  
A symbolic substitution $\sub$ is \defemph{primitive} if there exists some $n$ such that, for all $\lrx,\lry\in\A$, $\lrx$ occurs in $\sub^n(\lry)$.  

The \defemph{symbolic substitution space} $\SSS_\sub$ associated with a substitution $\sub$ is the set of all bi-infinite sequences of letters from $\A$ in which every finite subword occurs as a subword of $\sub^n (\lrx )$ for some $\lrx \in \A$ and some $n$.  

The substitution $\sub$ gives a set of tilings of $\R$ in the following way.  
Suppose $\A = \{ \lrx_1, \lrx_2, \ldots , \lrx_n\}$.  
The \defemph{substitution matrix} $A(\sub) = $ $(A_{ij}(\sub))$ of $\sub$ is the $n \times n$ matrix in which $A_{ij}(\sub)$ is the number of occurrences of $\lrx_i$ in $\sub (\lrx_j)$.  
Under the assumption that $\sub$ is primitive, some power of $A(\sub)$ contains strictly positive entries, so by the Perron-Frobenius theory, $A(\sub)$ has a leading eigenvalue $\inflation$ with a positive left eigenvector.  
To each letter $\lrx_i$, we can assign a tile---which is just an interval, the length of which is the $i$th entry of this left eigenvector.  
If two entries $i$ and $j$ of the eigenvector have the same value, let us distinguish between the associated tiles by labelling them $\lrx_i$ and $\lrx_j$.  
These labelled intervals are the tiles, and given a bi-infinite sequence in the symbolic substitution space of $\sub$, a tiling can be constructed in the obvious way, by adjoining these intervals end to end in the order specified by the sequence, with the origin located at the left endpoint of the first entry in the sequence.  
Then the \defemph{tiling space} $\TS_\sub$ associated with $\sub$ is the set of all translates of tilings constructed in this way from elements of the symbolic substitution space.  
The substitution $\sub$ acts on the set of tiles, and therefore also on $\TS_\sub$, by replacing each tile with a translated sequence of tiles, the total length of which is $\inflation$ times the length of the original tile.  
The purpose of choosing the left Perron-Frobenius eigenvector components as tile lengths is to guarantee that the tiles scale by the Perron-Frobenius eigenvalue under substitution.   

\begin{remark}
This discussion of substitution matrices and Perron-Frobenius theory is actually not limited to symbolic sequences, but applicable just as well to tile substitutions, even in higher dimensions (where tile length has to be replaced by area or volume).
\end{remark}

There is also a notion of a mixed symbolic substitution system.  
\begin{defn}\label{DEF:mixed-symbolic}
Let $\family = \{ \sub_1,\sub_2,\ldots , \sub_k\}$ be a finite set of symbolic substitutions all acting on the same alphabet $\A$, and consider an infinite sequence $\seq = (\seq_1, \seq_2,\seq_3,\ldots ) \in \{1, 2, \ldots , k\}^\N$.  
The \defemph{mixed symbolic substitution space} $\SSS_{\family,\seq}$ associated with $\family$ and $\seq$ is the set of all bi-infinite sequences of letters from $\A$ in which every finite subword occurs as a subword of $\sub_{\seq_1}\sub_{\seq_2}\cdots\sub_{\seq_n} (\lrx )$ for some $\lrx \in \A$ and some $n$.  
\end{defn}
\begin{remark}
Mixed symbolic substitution spaces are often referred to as \defemph{$s$-adic spaces} (see \cite{F:s-adic}, \cite{D:linearly-recurrent}).  
\end{remark}
\begin{remark}\alabel{REM:caveat}
It is not immediately clear that all mixed symbolic substitution spaces can be viewed as tiling spaces in the manner described above.  
In order to apply the topological techniques from the theory of tiling spaces (see Section \ref{SEC:ap-complex}) to the study of a mixed symbolic substitution space $\SSS_{\family,\seq}$, let us always assume that the substitution matrices $A(\sub_i)$ share a common left Perron-Frobenius eigenvector for all $1\leq i\leq k$, and therefore that the symbols $\lrx\in\A$ can be assigned well-defined tile lengths.  
This requirement might not be strictly necessary in order to apply the topological approach, but at the moment it is not completely clear that it can be dropped.  
\end{remark}

%% file: ap-complex.tex
\subsection{The Anderson-Putnam Complex}\alabel{SEC:ap-complex}

In \cite{AP} a method is given for computing the Cech cohomology with integer coefficients of a substitution tiling space $\TS_\sub$.  
The idea is that the dynamical system $(\TS_\sub, \sub)$ is topologically conjugate to a certain inverse limit space with a right shift map.  
This inverse limit space is defined in terms of a certain cell complex, which we will describe now.

For a tiling $\tiling$ and a point $u\in\R^d$, define 
\[
\tiling(u) = \{ \tile \in \tiling : u\in\tile\}.
\]
This definition can be extended to subsets $U$ of $\R^d$:
\[
\tiling(U) = \bigcup_{u\in U}T(u).
\]

\begin{defn}\alabel{DEF:APcomplex}
Let $\TS$ be a tiling space.  
Given a tile $\tile$ in some tiling $\tiling$ in $\TS$, the set $\tiling(\tile)$ is called a \defemph{collared tile}.  

The \emph{Anderson-Putnam complex} of a tiling space $\TS$ is denoted by $AP(\TS)$, and consists of a certain topological space under a certain equivalence relation.  
The topological space is $\TS \times \R^d$ under the product topology, where the topology on $\TS$ is the discrete topology and the topology on $\R^d$ is the standard topology.  
The equivalence relation is the smallest relation $\sim$ that equates $(\tiling_1,u_1)$ and $(\tiling_2,u_2)$ if $\tiling_1(\tile_1)-u_1 = \tiling_2(\tile_2)-u_2$ for some tiles $\tile_1,\tile_2$ with $u_1\in\tile_1\in\tiling_1$ and $u_2\in\tile_2\in\tiling_2$.  
\end{defn}

\begin{remark}
The Anderson-Putnam complex can be defined for any tiling space, but it is particularly useful when dealing with substitution and mixed substitution spaces.  
This is because these tiling spaces, which are relatively complicated objects, can be shown to be isomorphic to inverse limits of Anderson-Putnam complexes, which are relatively simple objects.  

The next three propositions are all proved in \cite{AP} for the class of substitution tiling spaces.  
The arguments in \cite{AP} can all be used with minimal modifications to prove these more general statements for mixed substitution spaces.  
\end{remark}
\begin{prop}[Proposition 4.2 in \cite{AP}]\alabel{PROP:complex-map}
The map $\sub_{s_i}$ induces a continuous map $\shift_i$ from $AP(\TS_{\family,\seqshift^i\seq})$ to $AP(\TS_{\family,\seqshift^{i-1}\seq})$ defined by $\shift_i (\tiling, u) = (\sub_{\seq_i}(\tiling), \lambda_{\seq_i} u)$.
\end{prop}

The next two theorems provide the necessary tools to compute the Cech cohomology with integer coefficients of $\TS_{\family,\seq}$.

\begin{thm}[Theorem 4.3 in \cite{AP}]
The space $\TS_{\family,\seq}$ is homeomorphic to the inverse limit space $\varprojlim_{\shift_i} AP(\TS_{\family,\seqshift^i\seq})$.  
\end{thm}
\begin{thm}[Theorem 6.1 in \cite{AP}]\alabel{THM:direct-limit}
The Cech cohomology group $H^j (\TS_{\family,\seq})$ is isomorphic to the direct limit of the system of abelian groups
\[
\minCDarrowwidth20pt%
\begin{CD}
H^j(AP(\TS_{\family,\seq})) @>\shift_1^*>> %
H^j(AP(\TS_{\family,\seqshift\seq})) @>\shift_2^*>> %
H^j(AP(\TS_{\family,\seqshift^2\seq})) @>\shift_3^*>> %
\cdots 
\end{CD}
\]
\end{thm}

This gives us a practical method to compute the cohomology of $\TS_{\family,\seq}$ in terms of $AP(\TS_{\family,\seqshift^i\seq})$ and the maps $\shift_i$.  

The focus of this work is on one-dimensional substitution tilings, so let us discuss the Anderson-Putnam complex for this class of tilings in more detail.  
In fact, in this case, only combinatorial information from $AP(\TS_{\family,\seq})$ and $\shift_n$ is used in this computation.  
Therefore, even though the lengths of the collared tiles $\tiling(\tile)$ must be known in order to define the maps $\shift_n$, the cohomology groups do not depend on this information.  
So, when dealing with one-dimensional systems, we will often work on a purely symbolic level and suppress any mention of the tiles in our discussion.  
Indeed, when dealing with one-dimensional substitution tiling spaces, let us suppress any mention of the tiling space $\TS_{\family,\seq}$, and instead, by an abuse of notation, let us speak of $AP(\family,\seq)$ and $H^i(\family,\seq)$ instead of $AP(\TS_{\family,\seq})$ and $H^i(\TS_{\family,\seq})$ (or, if $\family = \{\sub\}$, let us simply write $AP(\sub)$ and $H^i(\sub)$).  
Note that the interpretation of a mixed symbolic substitution system as a tiling space might be limited to the case with appropriately chosen tile lengths (see Remark \ref{REM:caveat}).  

In light of this discussion, there is a description of the Anderson-Putnam complex of a mixed symbolic substitution $(\family,\seq)$ that is much easier to visualize.  
This complex has the structure of a directed graph that contains one edge for each three-letter word $\lrx_1\lrx_2\lrx_3$ that appears as a subword of some iterated substitution of some letter.  
The head of the edge $\lrx_1\lrx_2\lrx_3$ is equal to the tail of the edge $\lry_1\lry_2\lry_3$ if $\lrx_2\lrx_3 = \lry_1\lry_2$ and the word $\lrx_1\lrx_2\lrx_3\lry_3$ appears as a subword of some iterated substitution of some letter.  
This vertex can be conveniently labelled with the overlap word $\lrx_2\lrx_3$.  
If there are multiple vertices corresponding to the same overlap word, we can distinguish them by using subscripts, as in Example \ref{EX:bbaaab-bbab}.

We can derive from this Anderson-Putnam complex three pieces of relevant combinatorial information that are used in the computation of cohomology: a coboundary matrix $\delta_{1,n}(AP(\family,\seq))$, and two matrices $A_{0,n}(AP(\family,\seq))$ and $A_{1,n}(AP(\family,\seq))$ that describe where the $0$-cells and $1$-cells of the complex $AP(\family,\seqshift^n\seq)$ are mapped under $\shift_n$.  
In particular, $\delta_{1,n}(AP(\family,\seq))$ has one row for each $1$-cell and one column for each $0$-cell of the complex $AP(\family,\seqshift^{n-1}\seq)$.  
Its entry at position $(i, j)$ is $1$ if vertex $j$ is the head of edge $i$, $-1$ if vertex $j$ is the tail of edge $i$, and $0$ otherwise.  
$A_{0,n}(AP(\family,\seq))$ is a square matrix, the rows and columns of which correspond to the $0$-cells of the complexes $AP(\family,\seqshift^n\seq)$ and $AP(\family,\seqshift^{n-1}\seq)$ respectively.  
Its entry at position $(i,j)$ is $1$ if $\shift_n$ sends vertex $i$ of $AP(\family,\seqshift^{n}\seq)$ to vertex $j$ of $AP(\family,\seqshift^{n-1}\seq)$ and $0$ otherwise.  
$A_{1,n}(AP(\family,\seq))$ is a square matrix, the rows and columns of which correspond to the $1$-cells of the complexes $AP(\family,\seqshift^n\seq)$ and $AP(\family,\seqshift^{n-1}\seq)$ respectively.  
Its entry at position $(i,j)$ is the number of times that the edge $j$ of $AP(\family,\seqshift^{n-1}\seq)$ appears in the image under $\shift_n$ of edge $i$ of $AP(\family,\seqshift^{n}\seq)$.  
More specifically, if the collared tiles $\lrx_1\lrx_2\lrx_3$ and $\lry_1\lry_2\lry_3$ are indexed by $i$ and $j$ respectively, then the $(i,j)$-th entry of $A_{1,n}(AP(\family,\seq))$ is the number of occurrences of the word $\lry_1\lry_2\lry_3$ as a subword of $\sub_{\seq_n}(\lrx_1\lrx_2\lrx_3)$ for which the middle letter $\lry_2$ occurs in the image $\sub_{\seq_n}(\lrx_2)$ of the middle letter $\lrx_2$.  

When the complex $AP(\family,\seq)$ is understood, let us write only $\delta_{1,n}, A_{0,n}$, and $A_{1,n}$ instead of $\delta_{1,n}(AP(\family,\seq))$, $A_{0,n}(AP(\family,\seq))$, and $A_{1,n}(AP(\family,\seq))$.  
Note that the matrix $\delta_{1,n}$ depends only on the structure of $AP(\family,$ $\seqshift^{n-1}\seq)$, but $A_{0,n}$ and $A_{1,n}$ depend also on the map $\shift_n$.  

Then the cohomology groups $H^0(\family,\seq)$ and $H^1(\family,\seq)$ are computed as follows.  
$H^0(\family,\seq)$ is the direct limit of 
\[
\begin{CD}
\ker (\delta_{1,n}(\family,\seq)) @>A_{0,1}>> %
\ker (\delta_{1,n}(\family,\seq)) @>A_{0,2}>> %
\ker (\delta_{1,n}(\family,\seq)) @>A_{0,3}>> %
\cdots 
\end{CD}
\]
and $H^1(\family,\seq)$ is the direct limit of
\[
\begin{CD}
\Z^{k_1} / (\delta_{1,n}(\family,\seq))(\Z^{l_2}) @>\tilde{A}_{1,1}>> %
\Z^{k_2} / (\delta_{1,n}(\family,\seq))(\Z^{l_3}) @>\tilde{A}_{1,2}>> %
\cdots 
\end{CD}
\]
where $k_n$ is the number of edges and $l_n$ the number of vertices in $AP(\family,\seqshift^{n-1}\seq)$, and $\tilde{A}_{1,n}$ is the matrix induced by $A_{1,n}$ on the quotient group $\Z^{k_{n}}/$ $(\delta_{1,n}(\family,\seq))(\Z^{l_{n+1}})$.  

\begin{remark}
If $\family = \{\sub\}$, then the matrices $\delta_{1,n}$, $A_{0,n}$, and $A_{1,n}$ do not depend on $n$, so we can simplify our notation and write $\delta_1$, $A_0$, and $A_1$.  
\end{remark}
\begin{example}\alabel{EX:bbaaab-bbab}
Figure \ref{FIG:complex-bbaaab-bbab} shows the Anderson-Putnam complex of the substition
\[
\begin{array}{rcl}
a & \to & bbaaab\\
b & \to & bbab
\end{array}
\]

\begin{figure}[h]
\begin{center}
\includegraphics[height=\complexheight]{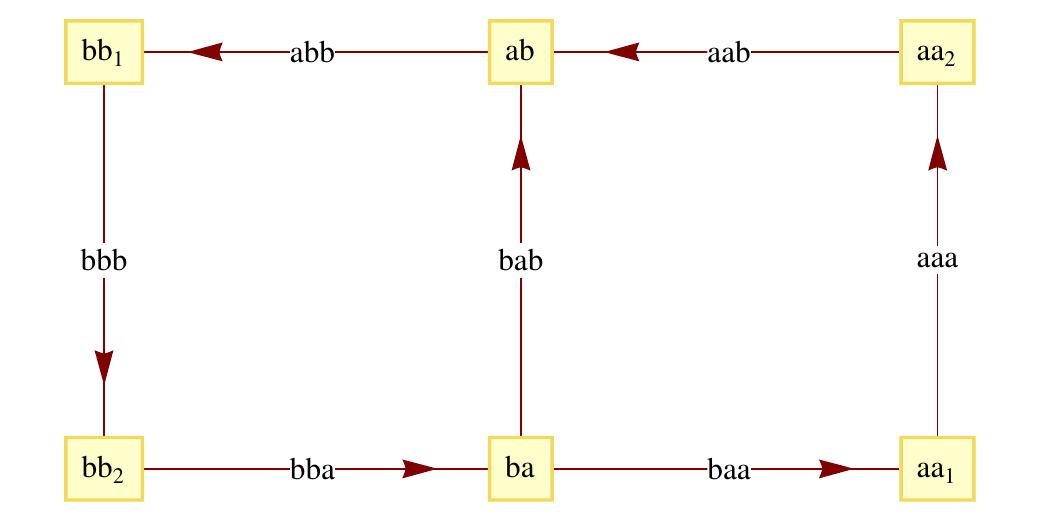}
\end{center}
\caption{$AP(  a \to bbaaab, b \to bbab )$\alabel{FIG:complex-bbaaab-bbab}}
\end{figure}

We can see by inspection that the word $aba$ will never occur in any iterated substitution of any letter, while two iterations of the substitution on the starting letter $a$ are sufficient to show that any other three-letter word is possible.
\[
\begin{array}{l}
a \\
bbaaab \\
bbabbbabbbaaabbbaaabbbaaabbbab
\end{array}
\]

The matrices are
\[
\delta_1 = \left[ 
\begin{array}{rrrrrr}
-1 &  1 &  0 &  0 &  0 &  0 \\
 0 & -1 &  1 &  0 &  0 &  0 \\
 0 &  0 & -1 &  0 &  1 &  0 \\
 1 &  0 &  0 & -1 &  0 &  0 \\
 0 &  0 &  1 & -1 &  0 &  0 \\
 0 &  0 &  0 &  1 &  0 & -1 \\
 0 &  0 &  0 &  0 & -1 &  1 
\end{array} \right],
\]
\[
A_0 = \left[
\begin{array}{cccccc}
0 & 0 & 0 & 0 & 1 & 0 \\
0 & 0 & 0 & 0 & 1 & 0 \\
0 & 0 & 0 & 0 & 1 & 0 \\
0 & 0 & 0 & 0 & 1 & 0 \\
0 & 0 & 0 & 0 & 1 & 0 \\
0 & 0 & 0 & 0 & 1 & 0 
\end{array}
\right],
\quad A_1 = \left[
\begin{array}{ccccccc}
1 & 1 & 1 & 1 & 0 & 1 & 1 \\
1 & 1 & 1 & 1 & 0 & 1 & 1 \\
0 & 0 & 1 & 0 & 1 & 1 & 1 \\
1 & 1 & 1 & 1 & 0 & 1 & 1 \\
1 & 1 & 1 & 1 & 0 & 1 & 1 \\
0 & 0 & 1 & 0 & 1 & 1 & 1 \\
0 & 0 & 1 & 0 & 1 & 1 & 1 
\end{array}
\right].
\]

The cohomology groups are
\[
H^0\left(\begin{array}{rcl}a & \to & bbaaab \\  b & \to & bbab \end{array} \right) %
\cong \Z, \ %
H^1\left(\begin{array}{rcl}a & \to & bbaaab \\  b & \to & bbab \end{array} \right) %
\cong \Z \left[ \frac{1}{6} \right] \oplus \ \Z \left[ \frac{1}{6}\right].%
\]

\end{example}

%% file: full-complex3.tex
\section{Changing the Underlying Cell Complex}\alabel{SEC:fullcomplex}

The Anderson-Putnam complex defined above depends on the particular substitution, and varies along a sequence of substitutions.  
In order to deal with mixed substitution systems, it will be useful to modify the Anderson-Putnam complex in such a way that it will work for many substitutions at the same time.  
This motivates the following definition.  

\begin{defn}
The \defemph{full Anderson-Putnam complex} on an alphabet $\A$, denoted by $AP(\A)$, is the directed graph defined as follows.

\begin{itemize}
\item  The vertices of $AP(\A)$ consist of all words $\lrx_1\lrx_2 \in \A^2$.
\item  The edges of $AP(\A)$ consist of all words $\lrx_1\lrx_2\lrx_3 \in \A^3$.
\item  The head of the edge $\lrx_1\lrx_2\lrx_3$ is the vertex $\lrx_2\lrx_3$ and its tail is the vertex $\lrx_1\lrx_2$.  
\end{itemize}
\end{defn}

The full complex on $\{ a, b\}$ is depicted in Figure \ref{FIG:fullcomplex-ab}.

\begin{figure}[h]
\begin{center}
\includegraphics[height=\fullcomplexheight]{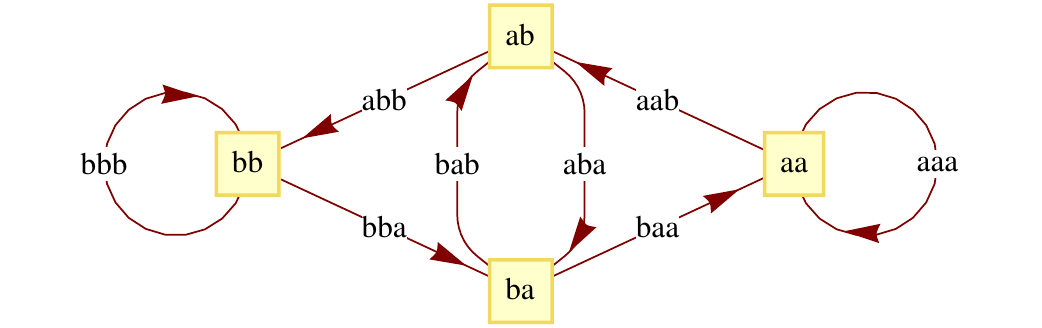}
\end{center}
\caption{$AP(\{ a, b\} )$\alabel{FIG:fullcomplex-ab}}
\end{figure}

Not every substitution on the alphabet $\A$ has the full complex of $\A$ as its Anderson-Putnam complex.  
Indeed, clearly $AP(a\to bbaaab, b\to bbab)$ in Example \ref{EX:bbaaab-bbab}, Figure \ref{FIG:complex-bbaaab-bbab} is different from $AP(\{ a,b\})$ in Figure \ref{FIG:fullcomplex-ab}: it has more edges, but fewer vertices.  
Nevertheless, the substitution $a\to bbaaab, b\to bbab$ induces a continuous map on $AP(\{ a,b\})$ in the usual way, and therefore it also induces a map on $H^i(AP(\{ a,b\}))$, and it is natural to ask if the inductive limits of these cohomology groups give the same answer as if we computed $\varinjlim H^i(AP(a\to bbaaab,b\to bbab))$.  
For certain substitutions, the answer to this question will be no, but for the substitution in Example \ref{EX:bbaaab-bbab}, the answer is yes; moreover, it is even true that the full Anderson-Putnam complex gives the same result at the level of topological spaces; that is,  
\[
\varprojlim_{\shift^*}AP(\sub) \cong \varprojlim_{\shift^*}AP(\{ a,b\}).
\]

The full Anderson-Putnam complex $AP(\{ a,b\} )$ differs from $AP(a\to bbaaab, b\to bbab)$ in two ways: it contains the extra edge $aba$, and it contains the vertices $aa$ and $bb$, which in $AP(a\to bbaaab, b\to bbab)$ have been split into $aa_1$, $aa_2$ and $bb_1$, $bb_2$ respectively.  
In the rest of this section, let us discuss the conditions under which one may modify the Anderson-Putnam complex while leaving topological invariants---either the cohomology groups or the topological space itself---unchanged.

\input{merge-vertices3}
\input{add-cells2}
\input{left-collar}

%% file: merge-vertices3.tex
\subsection{Merging Cells}\alabel{SEC:merge-cells}

Let us first show that the operation of merging vertices does not change the projective limit of the complexes.  

In fact, this is true more generally for any mixed substitution spaces that satisfy conditions 1--4 from Remark \ref{REM:mixed-conditions}.  
This proof involves defining a new, modified version of the Anderson-Putnam complex.  
\begin{defn}\alabel{DEF:modified-APcomplex}
Let $\TS$ be a tiling space, the tilings of which have a $CW$-structure.  

The \emph{modified Anderson-Putnam complex} of $\TS$ is denoted by $AP'(\TS)$, and consists of the same topological space $\TS\times \R^d$ from Definition \ref{DEF:APcomplex} under a different equivalence relation.  
The equivalence relation is the smallest relation $\approx$ that equates $(\tiling_1,u_1)$ and $(\tiling_2,u_2)$ if $u_1$ lies in a subcell $\cellc_1$ of a tile $\tile_1 \in \tiling_1$, $u_2$ lies in a subcell $\cellc_2$ of a tile $\tile_2\in \tiling_2$, and $\tiling_1(\cellc_1) - u_1 = \tiling_2(\cellc_2)-u_2$.  
\end{defn}

Let us denote the modified Anderson-Putnam complex of a one-dimensional mixed substitution space $\TS_{\family,\seq}$ by $AP'(\family,\seq)$, or by $AP'(\sub)$ if $\family = \{ \sub\}$.  

\begin{prop}\alabel{PROP:modified-shift}
The substitution $\sub_{s_i}$ induces a continuous map $\shift_i'$ from $AP'(\TS_{\family,\seqshift^i \seq})$ onto $AP'(\TS_{\family,\seqshift^{i-1}\seq})$ defined by the formula $\shift_i'(\tiling,u) = $ $(\sub_{s_i}(\tiling),\inflation_{s_i} u)$.  
\end{prop}
\begin{proof}
The proof is the same as that of Proposition \ref{PROP:complex-map} (Proposition 4.2 in \cite{AP}).
\end{proof}

\begin{example}\alabel{EX:modified-APcomplex}
The modified Anderson-Putnam complex of the substitution from Example \ref{EX:bbaaab-bbab} appears in Figure \ref{FIG:modified-bbaaab-bbab}.  
\begin{figure}[h]
\begin{center}
\includegraphics[height=\fullcomplexheight]{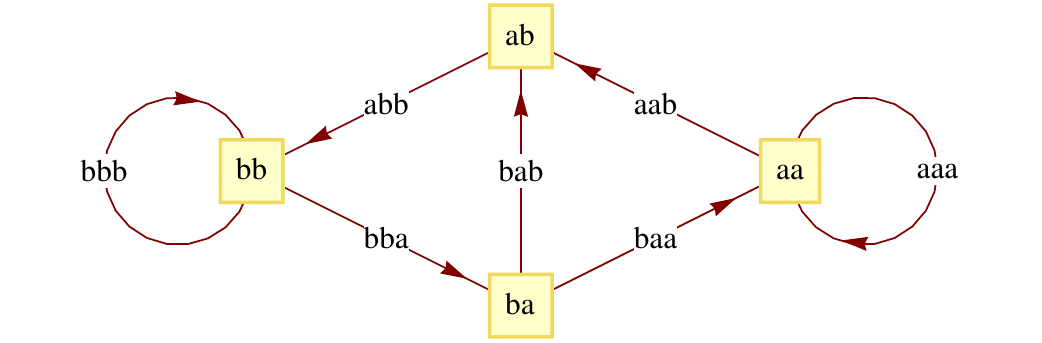}
\end{center}
\caption{$AP'( a\to bbaaab, b\to bbab )$\alabel{FIG:modified-bbaaab-bbab}}
\end{figure}
\end{example}

\begin{thm}\alabel{THM:merge-vertices}
Let $\TS_{\family,\seq}$ be a mixed substitution space of tilings of $\R^d$ satisfying the conditions 1--4 from Section \ref{REM:mixed-conditions}.  Then 
\[
\varprojlim_{\shift_i'}AP'(\TS_{\family,\seqshift^i\seq}) %
\cong %
\varprojlim_{\shift_i}AP(\TS_{\family,\seqshift^i\seq}) %
\]
\end{thm}
\begin{proof}
Let $X = \varprojlim_{\shift_i}AP(\TS_{\family,\seq})$ and $Y = \varprojlim_{\shift_i'}AP'(\TS_{\family,\seq})$, and let us view the elements of $X$ and $Y$ as sequences.  

The equivalence relation $\approx$ is coarser than the relation $\sim$, so there is a continuous quotient map $f_i : AP(\TS_{\family,\seqshift^{i-1}\seq})\to AP'(\TS_{\family,\seqshift^{i-1}\seq})$.  
Moreover, it is easy to see that $f_i \circ \shift_{i} = \shift_i' \circ f_{i+1}$.  

Let $F: X\to Y$ denote the continuous surjection induced by the family $\{ f_i\}$ at the level of projective limits.  
Let us prove the theorem by showing that $F$ is one-to-one.  

If we suppose that this is not the case, then there are two different sequences $(x_i^1)$ and $(x_i^2)$ in $X$ that are both mapped to the same sequence $(y_i) \in Y$ under $F$.  
Since $(x_i^1)$ and $(x_i^2)$ are different, there must be some index $N$ such that $x_i^1\neq x_i^2$ for all $i\geq N$.  
Let $(\tiling_i^1,u_i^1)$ and $(\tiling_i^2,u_i^2)$ be $\sim$-equivalence class representatives for $x_i^1$ and $x_i^2$ respectively.  
Then $(\tiling_i^1,u_i^1)\approx (\tiling_i^2,u_i^2)$, so the tiles touching some cell containing $u_i^1$ are translation equivalent to the tiles touching some cell containing $u_i^2$.  
In particular, this means that the tiles touching $u_i^1$ are translation equivalent to the tiles touching $u_i^2$; in other words, $\tiling_i^1(u_i^1) = \tiling_i^2(u_i^2)$.  
But for $i$ sufficiently large, $(\tiling_i^1,u_i^1)\nsim (\tiling_i^2,u_i^2)$, which means that $\tiling_i^1(\tiling_i^1(u_i^1))$ is not translationally equivalent to $\tiling_i^2(\tiling_i^2(u_i^2))$.  

The relations $\sim$ and $\approx$ are the same if $u_i^1$ and $u_i^2$ are both contained only in $d$-cells, but not in any cells of lower dimension.  
Therefore, for $i\geq N$, $u_i^1$ and $u_i^2$ must both be contained in lower-dimensional cells of their respective tilings.  

By FLC, for each dimension $l\leq d$, there exists a minimum distance $M_l$ defined by 
\begin{equation*}
\begin{split}
M_l & := \min %
\{ \epsilon > 0 : %
\tiling\in \TS_{\family,\seqshift^m\seq} \text{ for some } m \geq 1, \\
&  \qquad \cellc \text{ is an } l\text{-cell of }\tiling,  v\in \cellc, \text{ and } %
\overline{B_\epsilon (v)}\subseteq \tiling(\cellc)\}.
\end{split}
\end{equation*}

Let $M = \min_{l\leq d} M_l$, let $\lambda_0$ denote the minimum of the inflation factors of the substitutions $\sub_1, \ldots , \sub_k$, and define $L$ by 
\begin{equation*}
\begin{split}
L & := \max_{m\geq 1, \tiling\in\TS_{\family,\seqshift^m\seq}} \{ \diam (\tiling(t)) : t\in \tiling\},
\end{split}
\end{equation*}
which exists by FLC.  
Choose $N' \in \N$ such that $L < \lambda_0^{N'}M$.  
Then the partial tilings $\tiling_{N+N'}^1(u_{N+N'}^1)$ and $\tiling_{N+N'}^2(u_{N+N'}^2)$ agree up to translation and contain the open balls $B_M(u_{N+N'}^1)$ and $B_M(u_{N+N'}^2)$ respectively, so $\sub_{N}\cdots \sub_{N+N'-1} (\tiling_{N+N'}^1(u_{N+N'}^1))$ and $\sub_{N}\cdots \sub_{N+N'-1}$ $(\tiling_{N+N'}^2(u_{N+N'}^2))$ agree up to translation and contain the open balls $B_{\lambda_0^{N'}M}(u_{N+N'}^1)$ and $B_{\lambda_0^{N'}M}(u_{N+N'}^2)$.  
By our choice of $N'$, these open balls contain $\tiling_N^1(\tilet^1)$ and $\tiling_N^2(\tilet^2)$ respectively, for any tiles $\tilet^1$ and $\tilet^2$ with $u_N^1\in \tilet^1$ and $u_N^2 \in \tilet^2$.  
This means that $(\tiling_N^1,u_N^1)\sim (\tiling_N^2,u_N^2)$, which is a contradication.  
Therefore $F$ is one-to-one and $X\cong Y$.  
\end{proof}

%% file: add-cells2.tex
\subsection{Adding Cells}\alabel{SEC:add-cells}

The previous result is true for tilings in arbitrary dimension, but the results that follow will only be proved for tilings in one dimension.  
Therefore let us now assume that all tiling spaces are spaces of one-dimensional tilings.  

Let us show that, under certain conditions, it is possible to add cells to the modified Anderson-Putnam complex $AP'(\TS_{\family,\seq})$ without changing the projective limit.  

Sometimes the addition of a new cell to the complex changes the resulting projective limit.  
To see this, consider the next example, in which the cohomology groups of the projective limits are different.

\begin{example}\alabel{EX:aba-bbab}
The complex for the substitution
\[
\begin{array}{rcl}
a & \to & aba \\
b & \to & bbab
\end{array}
\]
appears in Figure \ref{FIG:complex-aba-bbab}.

\begin{figure}[h]
\begin{center}
\includegraphics[height=\complexheight]{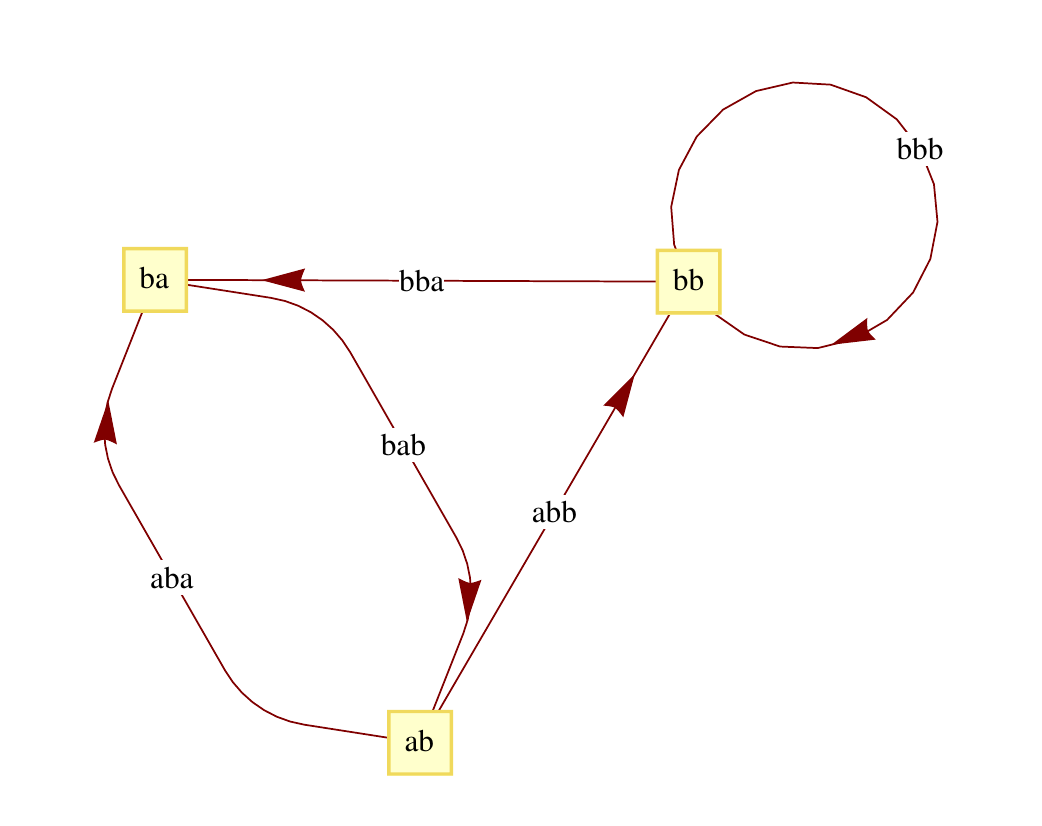}
\end{center}
\caption{$AP(a\to aba, b\to bbab )$\alabel{FIG:complex-aba-bbab}}
\end{figure}

The matrices are
\[
\delta_1 = \left[ 
\begin{array}{rrr}
-1 &  1 &  0 \\
-1 &  0 &  1 \\
 1 & -1 &  0 \\
 0 &  1 & -1 \\
 0 &  0 &  0 
\end{array} \right],
\]
\[
A_0 = \left[
\begin{array}{ccc}
1 & 0 & 0 \\
0 & 1 & 0 \\
0 & 0 & 1 
\end{array}
\right],
\quad A_1 = \left[
\begin{array}{ccccc}
1 & 1 & 1 & 1 & 0 \\
0 & 2 & 1 & 1 & 0 \\
1 & 0 & 2 & 0 & 0 \\
1 & 0 & 1 & 1 & 1 \\
0 & 1 & 1 & 1 & 1 
\end{array}
\right].
\]

The cohomology groups are
\[
H^0\left(\begin{array}{rcl}a & \to & aba \\  b & \to & bbab \end{array} \right) %
\cong \Z, \qquad %
H^1\left(\begin{array}{rcl}a & \to & aba \\  b & \to & bbab \end{array} \right) %
\cong \Z \left[ \frac{1}{5} \right] \oplus \ \Z \left[ \frac{1}{5}\right].%
\]

If, instead, we used $AP(\{ a,b\})$ to compute the cohomology groups, we would get different matrices
\[
\delta_1 = \left[ 
\begin{array}{rrrr}
 0 &  0 &  0 &  0 \\
-1 &  1 &  0 &  0 \\
 0 & -1 &  1 &  0 \\
 0 & -1 &  0 &  1 \\
 1 &  0 & -1 &  0 \\
 0 &  1 & -1 &  0 \\
 0 &  0 &  1 & -1 \\
 0 &  0 &  0 &  0 
\end{array} \right],
\]
\[
A_0 = \left[
\begin{array}{cccc}
1 & 0 & 0 & 0 \\
0 & 1 & 0 & 0 \\
0 & 0 & 1 & 0 \\
0 & 0 & 0 & 1 
\end{array}
\right],
\quad A_1 = \left[
\begin{array}{cccccccc}
0 & 1 & 1 & 0 & 1 & 0 & 0 & 0 \\
0 & 1 & 1 & 0 & 0 & 1 & 0 & 0 \\
0 & 0 & 1 & 1 & 0 & 1 & 1 & 0 \\
0 & 0 & 0 & 2 & 0 & 1 & 1 & 0 \\
0 & 0 & 1 & 0 & 1 & 1 & 0 & 0 \\
0 & 0 & 1 & 0 & 0 & 2 & 0 & 0 \\
0 & 0 & 1 & 0 & 0 & 1 & 1 & 1 \\
0 & 0 & 0 & 1 & 0 & 1 & 1 & 1 
\end{array}
\right].
\]
which result in the following cohomology groups.  
\[
H^0 %
\cong \Z, \qquad %
H^1 %
\cong \Z \left[ \frac{1}{5} \right] \oplus \ \Z \left[ \frac{1}{5}\right] \oplus \ \Z.%
\]

\end{example}

Example \ref{EX:aba-bbab} shows that it is not always possible to add cells to the complex without changing the resulting projective limits.  
Nevertheless, sometimes it is possible, and the question of exactly when it is possible motivates the following definitions.

\begin{defn}\alabel{DEF:self-correcting}
Let $\{ \sub_1,\ldots , \sub_k\}$ be a family of substitutions on an alphabet $\A$, and let $\seq = (\seq_1,\seq_2,\ldots ) \in \{ 1,\ldots ,k\}^\N$ be an infinite sequence.  
Then $(\family, \seq)$ is called \defemph{self-correcting} if there exists $n$ such that, for any $i_0\in \N$ and any two-letter subword $\lrx\lry$ of any word in $\sub_{i_0}\sub_{i_0+1}\cdots \sub_{i_i+n-1}(\A^2)$, there exist $m\in \N$ and $\lrz\in\A$ such that $\lrx\lry$ is a subword of $\sub_{i_0}\sub_{i_0+1}\cdots \sub_{i_0+m}(\lrz)$.  
\end{defn}

The substitution $a\to aba, b\to bbab$ from Example \ref{EX:aba-bbab} is not self-correcting because $aa$ appears in $\sub^n(aa)$ for every $n$, but does not appear in $\sub^m (a)$ or $\sub^m (b)$ for any $m$.  

The proof of the following proposition is straightforward.  
\begin{prop}\alabel{PROP:sc-more-cells}
If $(\family, \seq)$ is self-correcting, then for each $k\geq 1$ there exists $n$ such that, for any $i_0\in \N$ and any $k$-letter subword $\wdw$ of any word in $\sub_{i_0}\sub_{i_0+1}\cdots \sub_{i_0+n-1}(\A^*)$, there exist $m\in \N$ and $\lrz\in\A$ such that $\wdw$ is a subword of $\sub_{i_0}\sub_{i_0+1}\cdots \sub_{i_0+m}(\lrz)$.  
\end{prop}

The significance of the self-correcting condition is that, if $(\family, \seq)$ is self-correcting, then we may add cells to the modified Anderson-Putnam complex $AP'(\family,\seqshift^i\sub)$ without changing the resulting projective limit.  

\begin{thm}\alabel{THM:add-cells}
Let $\family = \{ \sub_1,\ldots , \sub_k\}$ be a family of substitutions on an alphabet $\A$, let $\seq = (\seq_1,\seq_2,\ldots )\in \{ 1,\ldots , k\}^\N$ be an infinite sequence, and suppose that $(\family,\seq)$ is self-correcting.  
Let $\shift_i'$ denote the map induced by $\sub_i$ on $AP'(\family,\seqshift^i \seq)$, and let $\shift_i$ denote the map induced on $AP(\A)$ by $\sub_i$.  
Then 
\[
\varprojlim_{\shift_i'}AP'(\family,\seqshift^i\seq) = \varprojlim_{\shift_i}AP(\A).
\]
\end{thm}
\begin{proof}
Let us write 
\[
\varprojlim_{\shift_i'}AP'(\family,\seqshift^i\seq) = \{ (x_1, x_2, \ldots ) \in \prod_{i\geq 0}AP'(\family,\seqshift^i\seq) : x_i = \shift_i' (x_{i+1}) \ \forall \  i\},
\]
and similarly for $\varprojlim_{\shift_i}AP(\A)$.  

Then $AP'(\family,\seqshift^i\seq)$ is a subspace of $AP(\A)$, so the identity map is a continuous injection of $\varprojlim_{\shift_i'}AP'(\family,\seqshift^i\seq)$ into $\varprojlim_{\shift_i}AP(\A)$.  
Let us show that this map is also surjective.  

Pick a sequence $(x_1, x_2, \ldots ) \in \varprojlim_{\shift_i}AP(\A)$.  
We know from Proposition \ref{PROP:sc-more-cells} that there is $n$ such that, if $\lrx\lry\lrz \in \A^3$ and $j\in\N$, then, for some $m\geq 0$, every three-letter subword of $\sub_{j}\sub_{j+1}\cdots \sub_{j+n-1}(\lrx\lry\lrz)$ occurs as a subword of some word in $\sub_{j}\sub_{j+1}\cdots \sub_{j+m-1}(\A)$.  
But this means that every three-letter subword of $\sub_{j}\sub_{j+1}\cdots \sub_{j+n-1}(\lrx\lry\lrz)$ is an edge in $AP'(\family,\seqshift^{j-1}\seq)$.  
Therefore $\shift_{j}\shift_{j+1}\cdots \shift_{j+n-1}$ sends $AP(\A)$ to $AP'(\family,\seqshift^{j-1}\seq)$.  
Therefore, since $x_{j+n}\in AP(\A)$, we must have $x_{j}\in AP'(\family,\seqshift^{j-1}\seq)$.  

This is true for all $j\in\N$, so $\varprojlim_{\shift_i'}AP'(\family,\seqshift^i\seq) = \varprojlim_{\shift_i}AP(\A)$.  
\end{proof}

\begin{remark}\alabel{REM:smaller-limit}
Even if $(\family, \seq)$ fails to be self-correcting, the space $AP'(\family,$ $\seqshift^i\seq)$ is a subcomplex of $AP(\A)$, and so $H^j(AP(\family,\seqshift^i\seq))$ is a subgroup of $H^j(AP(\A))$.  
Therefore, even if the spaces are not the same, at the level of cohomology we can say that $\varinjlim_{\shift_i'^*}H^j(AP'(\family,\seqshift^i\seq))$ is a subgroup of $\varinjlim_{\shift_i^*}H^j(AP(\A))$.  
\end{remark}

\begin{remark}\alabel{REM:not-conjugate}
Theorems \ref{THM:merge-vertices} and \ref{THM:add-cells} show that the modified Anderson-Putnam complex and the full Anderson-Putnam complex give rise to the same space as the ordinary Anderson-Putnam complex at the level of projective limits.  
In fact, more can be said: for a single substitution, the dynamical systems $(\TS_\sub,\sub)$, $(\varprojlim_{\shift} AP(\sub),\omega)$, $(\varprojlim_{\shift'}AP'(\sub),\omega')$, and $(\varprojlim_{\shift}AP(\A),\omega)$ are all topologically conjugate, where $\omega$ is the right shift map defined by $\omega (x)_i = \shift(x_i)$, and similarly for $\omega'$ (see \cite{AP}, Theorem 4.3).  
\end{remark}

%% file: left-collar.tex
\subsection{Left Collaring}\alabel{SEC:left-collar}

There is yet another simplification that can be made to the Anderson-Putnam complex of a one-dimensional tiling space.  
\begin{defn}\alabel{DEF:left-collared-complex}
Let $\A$ be an alphabet.  
The \defemph{left-collared Anderson-Putnam complex} of $\A$, denoted $AP_{\collar}(\A)$, is the complex obtained from $AP(\A)$ by identifying all edges $\lrx\lry\lrz$ and $\lru\lrv\lrw$ for which $\lrx\lry = \lru\lrv$, and also all vertices $\lrx\lry$ and $\lru\lrv$ for which $\lrx = \lru$.  
\end{defn}
\begin{remark}
The complex $AP_{\collar}(\A)$ has a very simple description.  
It is a directed graph, the edges of which are two-letter words $\lrx\lry\in\A^2$ and the vertices of which are letters $\lrz\in\A$.  
The head and tail of $\lrx\lry$ are $\lry$ and $\lrx$ respectively.  

Given a family $\family$ of substitutions on an alphabet $\A$ and an infinite sequence $\seq$, similar left-collared complexes $AP_\collar(\family,\seqshift^i\seq)$ and $AP_\collar'(\family,\seqshift^i\seq)$ can be constructed as quotients of $AP(\family,\seqshift^i\seq)$ and $AP'(\family,\seqshift^i\seq)$ respectively, and $\sub_{\seq_i}$ induces maps on all of these complexes.  

Of course, one could also define right-collared Anderson-Putnam complexes, and obtain for them results analogous to Proposition \ref{PROP:left-collared}.  
\end{remark}



Let us now introduce some notation relating to Cech cohomology.  
If $X$ is a topological space with the structure of a CW-complex, let $\cellc$ denote a cell in $X$, and let $\cochainc$ denote the corresponding cochain.  
Let $C^d(X)$ denote the group of $d$-cochains of $X$.  
If $\psi$ is a cellular map with domain $X$, let $\psi(d)$ denote the induced map on $d$-cells.  

The proof of Proposition \ref{PROP:left-collared} below relies on the theory of quotient cohomology, which is introduced in \cite{BS:quotient}.  
Before proving the proposition, let us review some of the relevant notions from this theory.  

The theory of quotient cohomology applies to topological spaces $X$ and $Y$ for which there is a quotient map $f : X \to Y$ such that the pullback $f^*$ is injective on cochains.  
In such a situation, the cochain group $C_Q^d(X,Y)$ is defined to be the quotient $C^d(X)/f^*(C^d(Y))$.  
The usual coboundary operator sends $C_Q^d(X,Y)$ to $C_Q^{d+1}(X,Y)$, and the \defemph{quotient cohomology} $H_Q^d(X,Y)$ is defined to be the kernel of the coboundary modulo the image.  
Then the short exact sequence of cochain complexes\\
\centerline{
\xymatrix{
0 \ar[r] & 
C^d(Y) \ar[r]^{f^*} & %
C^d(X) \ar[r] & %
C_Q^d(X,Y) \ar[r] & %
0
}
}
induces the long exact sequence\\
\centerline{
\xymatrix{
\cdots \ar[r] & %
H_Q^{d-1}(X,Y) \ar[r] & %
H^{d}(Y) \ar[r]^{f^*} & %
H^{d}(X) \ar[r] & %
H_Q^{d}(X,Y) \ar[r] & %
\cdots %
}
}

It is this long exact sequence that will enable us to show that $f^* : H^d(Y)\to H^d(X)$ is an isomorphism.

\begin{prop}\alabel{PROP:left-collared}
Let $\family = \{ \sub_1,\ldots,\sub_k\}$ be a family of substitutions on an alphabet $\A$, and let $\seq \in \{ 1, \ldots , k\}^\N$ be an infinite sequence.  
Let $AP_{\collar}'(\family,\seqshift^i\seq)$ and $AP_{\collar}(\A)$ denote the left-collared complexes obtained as quotients of $AP'(\family,\seqshift^i\seq)$ and $AP(\A)$ respectively.  
Let $\shift_{i,\collar}'$ and $\shift_{i,\collar}$ respectively denote the maps induced by $\sub_{\seq_i}$ on these complexes.  
Then the Cech cohomologies of $\varprojlim_{\shift_i'}AP'(\family,\seqshift^i\seq)$ and $\varprojlim_{\shift_{i,\collar}'}AP_{\collar}(\family,\seqshift^i\seq)$ are isomorphic.  
If $(\family, \seq)$ is self-correcting, then the Cech cohomologies of $\varprojlim_{\shift_i'}AP'(\family,\seqshift^i\seq)$ and $\varprojlim_{\shift_{i,\collar}}AP_{\collar}(\A)$ are also isomorphic.  
\end{prop}

\begin{proof}
By Theorem \ref{THM:add-cells}, if $(\family, \seq)$ is self-correcting, then the projective limits $\varprojlim_{\shift_i}AP(\A)$ and $\varprojlim_{\shift_i'}AP'(\family,\seqshift^i\seq)$ are isomorphic.  
Therefore we can prove the second conclusion of the proposition by showing that the cohomologies of $\varprojlim_{\shift_i}AP(\A)$ and $\varprojlim_{\shift_{i,\collar}}AP_{\collar}(\A)$ are isomorphic.  
Let us prove this, and then describe how to modify the proof to show that the cohomologies of $\varprojlim_{\shift_i'}AP'(\family,\seqshift^i\seq)$ and $\varprojlim_{\shift_{i,\collar}'}AP_{\collar}(\family,\seqshift^i \seq)$ are isomorphic.  

Let $X$ and $Y$ denote $AP(\A)$ and $AP_{\collar}(\A)$ respectively.  
To see that the cohomology groups are isomorphic, note that there is a continuous quotient map $f: AP(\A) \to AP_{\collar}(\A)$, and the pullback of this map is injective on cochains.  
Then this quotient map gives rise to an exact sequence at the level of cohomology.  
The following diagram depicts the direct limit of this exact sequence.  

\centerline{
\xymatrix{
0 \ar[r] & %
H^{0}(Y) \ar[r]^{f^*} \ar[d]^{\shift_{1,\collar}^*(0)} & %
H^{0}(X) \ar[r] \ar[d]^{\shift_{1}^*(0)} & %
H_Q^{0}(X,Y) \ar[r] \ar[d]^{\tilde{\shift}_{1}^*(0)} & %
H^{1}(Y) \ar[r]^{f^*} \ar[d]^{\shift_{1,\collar}^*(1)} & %
H^{1}(X) \ar[r] \ar[d]^{\shift_1^*(1)} & %
H_Q^{1}(X,Y) \ar[r] \ar[d]^{\tilde{\shift}_1^*(1)} & %
0 \\%
0 \ar[r] & %
H^{0}(Y) \ar[r]^{f^*} \ar[d]^{\shift_{2,\collar}^*(0)} & %
H^{0}(X) \ar[r] \ar[d]^{\shift_2^*(0)} & %
H_Q^{0}(X,Y) \ar[r] \ar[d]^{\tilde{\shift}_{2}^*(0)} & %
H^{1}(Y) \ar[r]^{f^*} \ar[d]^{\shift_{2,\collar}^*(1)} & %
H^{1}(X) \ar[r] \ar[d]^{\shift_2^*(1)} & %
H_Q^{1}(X,Y) \ar[r] \ar[d]^{\tilde{\shift}_2^*(1)} & %
0 \\%
& \vdots %
& \vdots %
& \vdots %
& \vdots %
& \vdots %
& \vdots %
& %
}
}

Let us show that $H_Q^0(X,Y) = 0$ and that the column

\centerline{
\xymatrix{
H_Q^{1}(X,Y)  \ar[d]^{\tilde{\shift}_{0}^*(1)} \\%
H_Q^{1}(X,Y)   \ar[d]^{\tilde{\shift}_{1}^*(1)} \\%
 \vdots %
}
}
has inductive limit $0$; this will suffice to show that $\varinjlim_{\shift_i^*(j)}H^j(X) \cong \varinjlim_{\shift_{i,\collar}^*(j)}H^j(Y)$.  

To show that $H_Q^0(X,Y) = 0$, let us show that every cocycle in $C^0(X)$ lies in the image $f^*(C^0(Y))$.  

$X$ is strongly connected, so the cocyles in $C^0(X)$ are generated by the cochain
\[
\sum_{\lru\lrv\in\A^2}(\lru\lrv)'.
\]

If $\lru\in\A$, then the image under $f^*$ of the cochain $\lru'$ is
\[
(\lru *)' := \sum_{\lrv\in\A}(\lru\lrv)'.
\]
But then
\[
\sum_{\lru\lrv\in\A^2}(\lru\lrv)' = \sum_{\lru\in\A}\sum_{\lrv\in\A}(\lru\lrv)' = \sum_{\lru\in\A}(\lru *)'.
\]

Therefore $H_Q^0(X,Y) = 0$.  

In order to show that $\varinjlim_{\tilde{\shift}_i^*(1)}H_Q^1(X,Y) = 0$, let us pick a particular $1$-cell $\lrx\lry\lrz$ in $X$ and show that, modulo coboundaries in $C^1(X)$, the cochain $(\shift_i^*(1))((\lrx\lry\lrz)')$ lies in $f^*(C^1(Y))$.   
This argument requires the assumption that $|\sub_{\seq_i}(\lrv)| > 1$ for all $\lrv\in\A$, which we can always guarantee by passing to the composition of sufficiently many substitutions $\sub_{\seq_i}\sub_{\seq_{i+1}}\cdots \sub_{\seq_{i+n}}$.    
Then we may apply our argument to $(\shift_{i+n}^*(1)\shift_{i+n-1}^*(1)\cdots \shift_i^*(1))((\lrx\lry\lrz)')$ instead of $(\shift_i^*(1))((\lrx\lry\lrz)')$.  

First let us describe the subgroup $f^*(C^1(Y))$.  
Given a $1$-cell $\lru\lrv$ of $Y$, we have that
\[
f^*((\lru\lrv)') = (\lru\lrv *)' := \sum_{\lrw\in\A}(\lru\lrv\lrw)'.
\]
Then $f^*(C^1(Y))$ is the subgroup spanned by all such cochains.  

For each $\lru\lrv\lrw\in\A^3$, let $N_{\lru\lrv\lrw}$ denote the number of occurrences of the $1$-cell $\lrx\lry\lrz$ in $\shift_{\seq_i}(\lru\lrv\lrw)$.  
Then
\[
(\shift_i^*(1))((\lrx\lry\lrz)') = \sum_{\lru\lrv\lrw\in\A^3}N_{\lru\lrv\lrw}(\lru\lrv\lrw)'.
\]

The argument relies on a decomposition of $N_{\lru\lrv\lrw}$.  
For each $\lrv\in\A$, let $n_{\lrv}$ denote the number of occurrences of the word $\lrx\lry\lrz$ in $\sub_{\seq_i}(\lrv)$.  
Then
\[
N_{\lru\lrv\lrw} = n_{\lrv} %
+ \delta_{\sub_{\seq_i}(\lru)_{[-1]}\sub_{\seq_i}(\lrv)_{[1,2]},\lrx\lry\lrz} %
+ \delta_{\sub_{\seq_i}(\lrv)_{[-2,-1]}\sub_{\seq_i}(\lrw)_{[1]},\lrx\lry\lrz}.%
\]

Therefore
\begin{align*}
(\shift_i^*(1))((\lrx\lry\lrz)') & = \sum_{\lru\lrv\lrw\in\A^3}N_{\lru\lrv\lrw}(\lru\lrv\lrw)' \\%
 & = \sum_{\lru\lrv\lrw\in\A^3}  n_{\lrv} (\lru\lrv\lrw)'%
+ \sum_{\lru\lrv\lrw\in\A^3} \delta_{\sub_{\seq_i}(\lru)_{[-1]}\sub_{\seq_i}(\lrv)_{[1,2]},\lrx\lry\lrz} (\lru\lrv\lrw)' \\%
& \qquad + \sum_{\lru\lrv\lrw\in\A^3} \delta_{\sub_{\seq_i}(\lrv)_{[-2,-1]}\sub_{\seq_i}(\lrw)_{[1]},\lrx\lry\lrz} (\lru\lrv\lrw)' \\%
 & = \sum_{\lru\lrv\in\A^2}\sum_{\lrw\in\A}  n_{\lrv} (\lru\lrv\lrw)'%
+ \sum_{\lru\lrv\in\A^2}\sum_{\lrw\in\A} \delta_{\sub_{\seq_i}(\lru)_{[-1]}\sub_{\seq_i}(\lrv)_{[1,2]},\lrx\lry\lrz} (\lru\lrv\lrw)' \\%
& \qquad + \sum_{\lrv\lrw\in\A^2}\sum_{\lru\in\A} \delta_{\sub_{\seq_i}(\lrv)_{[-2,-1]}\sub_{\seq_i}(\lrw)_{[1]},\lrx\lry\lrz} (\lru\lrv\lrw)' \\%
 & = \sum_{\lru\lrv\in\A^2}  n_{\lrv}\sum_{\lrw\in\A} (\lru\lrv\lrw)'%
+ \sum_{\lru\lrv\in\A^2} \delta_{\sub_{\seq_i}(\lru)_{[-1]}\sub_{\seq_i}(\lrv)_{[1,2]},\lrx\lry\lrz}\sum_{\lrw\in\A} (\lru\lrv\lrw)' \\%
& \qquad + \sum_{\lrv\lrw\in\A^2} \delta_{\sub_{\seq_i}(\lrv)_{[-2,-1]}\sub_{\seq_i}(\lrw)_{[1]},\lrx\lry\lrz}\sum_{\lru\in\A} (\lru\lrv\lrw)' \\%
 & = \sum_{\lru\lrv\in\A^2}  n_{\lrv}(\lru\lrv *)' %
+ \sum_{\lru\lrv\in\A^2} \delta_{\sub_{\seq_i}(\lru)_{[-1]}\sub_{\seq_i}(\lrv)_{[1,2]},\lrx\lry\lrz} (\lru\lrv *)' \\%
& \qquad + \sum_{\lrv\lrw\in\A^2} \delta_{\sub_{\seq_i}(\lrv)_{[-2,-1]}\sub_{\seq_i}(\lrw)_{[1]},\lrx\lry\lrz}\sum_{\lru\in\A} (\lru\lrv\lrw)' \displaybreak[0] \\%
\intertext{But $\sum_{\lru\in\A}(\lru\lrv\lrw)' - \sum_{\lru\in\A}(\lrv\lrw\lru)'$ is the image under the coboundary map $\delta$ of the $0$-cochain $(\lrv\lrw)'$, so modulo coboundaries we get}%
(\shift_i^*(1))((\lrx\lry\lrz)') & = \sum_{\lru\lrv\in\A^2}  n_{\lrv}(\lru\lrv *)' %
+ \sum_{\lru\lrv\in\A^2} \delta_{\sub_{\seq_i}(\lru)_{[-1]}\sub_{\seq_i}(\lrv)_{[1,2]},\lrx\lry\lrz} (\lru\lrv *)' \\%
& \qquad + \sum_{\lrv\lrw\in\A^2} \delta_{\sub_{\seq_i}(\lrv)_{[-2,-1]}\sub_{\seq_i}(\lrw)_{[1]},\lrx\lry\lrz}\sum_{\lru\in\A} (\lrv\lrw\lru)' \\%
 & = \sum_{\lru\lrv\in\A^2}  n_{\lrv}(\lru\lrv *)' %
+ \sum_{\lru\lrv\in\A^2} \delta_{\sub_{\seq_i}(\lru)_{[-1]}\sub_{\seq_i}(\lrv)_{[1,2]},\lrx\lry\lrz} (\lru\lrv *)' \\%
& \qquad + \sum_{\lrv\lrw\in\A^2} \delta_{\sub_{\seq_i}(\lrv)_{[-2,-1]}\sub_{\seq_i}(\lrw)_{[1]},\lrx\lry\lrz}(\lrv\lrw *)'.
\end{align*}

This shows that the cohomologies of $\varprojlim_{\shift_{i,\collar}}AP_\collar(\A)$ and $\varprojlim_{\shift_i}AP(\A)$ are isomorphic.  
To prove the analogous statement for $\varprojlim_{\shift_{i,\collar}'}AP_\collar'(\family,$ $\seqshift^i\seq)$ and $\varprojlim_{\shift_i'}AP'(\family,$ $\seqshift^i\seq)$, we can repeat the same argument, but everywhere we must alter our notation to restrict to $0$-cells and $1$-cells that are actually part of the complexes $AP_\collar'(\family,\seqshift^i\seq)$ and $AP'(\family,\seqshift^i\seq)$.  
\end{proof}

As an immediate consequence of Theorems \ref{THM:direct-limit}, \ref{THM:merge-vertices}, and \ref{THM:add-cells} and Proposition \ref{PROP:left-collared}, we can see that, if a mixed symbolic substitution system $(\family,\seq)$ is self-correcting, it is possible to compute the cohomology groups of its tiling space $\TS_{\family,\seq}$ as direct limits of cohomology groups of $AP_\collar(\A)$ with respect to the bonding maps induced by $\sub_{\seq_i}$.  

\begin{cor}\alabel{COR:cohomology-computation}
Let $\family = \{ \sub_1,\ldots,\sub_k\}$ be a family of substitutions on an alphabet $\A$, and let $\seq \in \{ 1,\ldots , k\}^\N$ be an infinite sequence.  
Let $AP_\collar'(\family,\seqshift^i\seq)$ and $AP_\collar(\A)$ denote the left-collared Anderson-Putnam complexes obtained as quotients of $AP'(\family,\seqshift^i\seq)$ and $AP(\A)$ respectively (see Definition \ref{DEF:left-collared-complex}), and let $\shift_{i,\collar}'$ and $\shift_{i,\collar}$ respectively denote the maps induced by $\sub_{\seq_i}$ on these complexes.  
Then 
\[
H^j(\family,\seq) \cong \varinjlim_{\shift_{i,\collar}'^*} H^j(AP_\collar'(\family,\seqshift^i\seq)).
\]

If $(\family,\seq)$ is self-correcting, then 
\[
H^j(\family,\seqshift^i\seq) \cong \varinjlim_{\shift_{i,\collar}^*} H^j(AP_\collar(\A)).
\]
\end{cor}

\begin{remark}\alabel{REM:partial-collar}
Dropping to the left-collared complex is really a form of partial collaring, which is described in \cite{S:book}.  
Corollary \ref{COR:cohomology-computation} is really saying that left-collaring works for all one-dimensional subsititution tilings simultaneously at the level of cohomology.  
\end{remark}
\begin{remark}\alabel{REM:b-d}
Theorems \ref{THM:merge-vertices} and \ref{THM:add-cells} describe a way to universalize the Anderson-Putnam complex for one-dimensional mixed substitution tiling spaces over a common alphabet, and Proposition \ref{PROP:left-collared} describes how to simplify the universalized complex.   
There exists another simplified version of the Anderson-Putnam complex, the Barge-Diamond complex (see \cite{BD}), which is dual to our one-sided complex, and thus closely related.  
It is conceivable that also for the Barge-Diamond complex a universal version can be constructed.  
\end{remark}

\begin{remark}\alabel{REM:sf-ubiquity}

If $|\family |=1$, there is actually no need for self-correcting.  
One could just as well work with the classical Anderson-Putnam complex (or a simplification thereof).  
Where it gets interesting is the case $|\family|>1$, if the Anderson-Putnam complexes of the individual substitutions differ.  
It is in this case that we want to use a larger, common complex, and need the self-correcting condition.  
But even in this case, we may find a complex smaller than the full one, that can still accommodate all substitutions involved, and is invariant under these substitutions.  
In that case, we only need self-correcting with respect to this smaller common complex, which is a weaker requirement.  

The next example shows that the self-correcting condition is not well-behaved; in particular, the class of self-correcting substitutions on $\A$ is not closed under composition.
\end{remark}
\begin{example}\alabel{EX:composition-not-sf}
Consider the two substitutions 
\[
\begin{array}{rcl}
a & \to & bab \\
b & \to & cbc \\
c & \to & cac 
\end{array}
\qquad \textrm{ and } \qquad 
\begin{array}{rcl}
a & \to & cac \\
b & \to & aba \\
c & \to & cbc 
\end{array}.
\]

Both of these substitutions are self-correcting, but the substitutions obtained by composing them (in either order) are not self-correcting.  
\end{example}

%% file: rank-bound3.tex
\section{The Rank and Structure of $H^1$}\alabel{SEC:rank-bound}

Let $(\family,\seq)$ be a self-correcting mixed symbolic substitution system on an alphabet $\A$.  
In this section, let us discuss the structure of $H^1(\TS_{\family,\seq})$.  

\subsection{The Rank of $H^1$}

In order to compute the rank of $H^1$, it will be convenient to think in terms of $\Q$-tensor products.  
Therefore, let us think of cochains, cocycles, and coboundaries as elements of the rational vector space $C^j(AP(\A)) \otimes \Q$.  

With this point of view, it is clear from Definition \ref{DEF:left-collared-complex} that the group $C^1(AP_\collar(\A))$ generates a rational vector space of dimension $n^2$, where $n = |\A|$.  
The following proposition says that the quotient space $C^1(AP_\collar(\A))\otimes \Q$ modulo the subspace generated by coboundaries has dimension $n^2 - n + 1$.  
\begin{prop}\alabel{PROP:rank-bound}
Let $\A$ be an alphabet, let $n$ denote $|\A|$, and let $AP_\collar(\A)$ denote the left-collared Anderson-Putnam complex of $\A$.  
Then the image of the coboundaries in $C^1(AP_\collar(\A))$ has rational dimension $n - 1$; that is,
\[
\dim_\Q(\delta(C^0(AP_\collar(\A)))\otimes \Q)= n - 1.
\]
\end{prop}
\begin{proof}
Given a $0$-cell $\lrv\in AP_\collar(\A)$, the image of the corresponding cochain under $\delta$ is 
\[
\delta(\lrv') = \sum_{\lru\in\A}(\lru\lrv)' - \sum_{\lrw\in\A}(\lrv\lrw)'.
\]

The elements $\{\delta(\lrv') : \lrv\in\A\}$ generate $\delta(C^0(AP_\collar(\A)))\otimes\Q$, so this subspace has dimension at most $n$.  

The existence of the relation
\[
\sum_{\lrv\in\A}\delta(\lrv') = \sum_{\lrv\in\A}\sum_{\lru\in\A}(\lru\lrv)' - \sum_{\lrv\in\A}\sum_{\lrw\in\A} (\lrv\lrw)' = 0
\]
means that this subspace has dimension no greater than $n - 1$.  
In fact, up to scalar multiplication, this is the only relation between elements of $\delta(C^0(AP_\collar(\A)))$.  
This is because the kernel of $\delta$ is generated by $\sum_{\lrv\in\A}\lrv'$.  
(For a general directed graph $G$, the kernel of $\delta$ is generated by elements of the form $\sum_{v\in C}v'$, where $C$ is a connected component of $G$.  
Since $AP_\collar(\A)$ is strongly connected, in this case the kernel of $\delta$ is singly-generated.)  

Since there is, up to scalar multiplication, only one non-trivial relation in $\spa_\Q\{ \delta(\lrv') : \lrv\in\A\}$, this space has dimension exactly $n - 1$.
\end{proof}
\begin{cor}\alabel{COR:rank-bound}
Let $(\family,\seq)$ be a mixed symbolic substitution system on an alphabet $\A$, and let $n$ denote $|\A|$.  
Then the rank of $H^1(\TS_{\family,\seq})$ is less than or equal to $n^2 - n + 1$.
\end{cor}
\begin{proof}
Corollary \ref{COR:cohomology-computation}  implies that the matrix $A_1(\sub_{\seq_i})$ of the map induced by $\sub_{\seq_i}$ on $C^1(AP_\collar(\A))$ has size $n^2 \times n^2$, and Proposition \ref{PROP:rank-bound} implies that this size drops down to $(n^2 - n + 1)\times(n^2 - n + 1)$ upon taking the quotient modulo coboundaries.  
Remark \ref{REM:smaller-limit} implies that $H^1(\TS_{\family,\seq})$ is a subgroup of $\varinjlim_{\shift_{i,\collar}^*}(H^1(AP_\collar(\A)))$.  
\end{proof}

The next example shows that the upper bound given in Corollary \ref{COR:rank-bound} is tight; that is, for each $n \geq 1$ and each alphabet $\A$ with size $n$, there exists a substitution $\sub_\A$ for which $H^1(\TS_{\sub_\A})$ has rank $n^2-n+1$.  

\begin{example}\alabel{EX:maximum-rank}
Let $\A = \{\lrx_0,\ldots,\lrx_{n-1}\}$ be an alphabet with $n$ letters.  
Let $\sub_\A$ be the substitution defined by
\[
\sub (\lrx_i) = \left\{ \begin{array}{ll}\lrx_i\lrx_{i+1} & \text{ if } i < n - 1\\
\lrx_{n-1}\lrx_0\lrx_0 & \text{ if } i = n - 1\end{array}\right. .
\]
\end{example}

\begin{prop}\alabel{PROP:maximum-rank}
The rank of $H^1(\TS_{\sub_\A})$ is $n^2-n+1$.
\end{prop}
\begin{proof}
Throughout the proof, let us perform addition and subtraction of indices modulo $n$.  

First let us observe that $\sub_\A$ is self-correcting.  
To see this, pick $\lrx_i\lrx_j \in \A^2$, and let $m = i - j - 1$.  
Then $\lrx_i\lrx_j$ are the $2^m$th and $(2^m+1)$st letters of $\sub_\A^{m+1} (\lrx_{j-1})$, so $\sub_\A$ is self-correcting because the set of two-letter words that occur as subwords in the iterated substitution of a single letter is the whole set $\A^2$.  


Let $\shift_\collar$ denote the cell map induced on $AP_\collar(\A)$ by $\sub_\A$, and let $A_1$ denote the matrix that describes the cochain map $\shift_\collar^* : C^1(AP_\collar(\A)) \to C^1(AP_\collar(\A))$.  
Then $A_1$ is an $n\times n$ matrix, the rows and columns of which can be indexed by elements of $\A^2$.  
$A_{\lrx\lry, \lru\lrv}$ is the number of occurrences of the $1$-cell $\lru\lrv$ in $\shift_\collar (\lrx\lry)$.  

Let us show that $H^1(\TS_{\sub_\A})$ has maximum possible rank by showing that the matrix $A_1$ is non-singular.  
If $\lrx_i\lrx_j$ is not one of the two-letter words $\lrx_k\lrx_{k+1}$ or $\lrx_0\lrx_0$, then $\lrx_i\lrx_j$ occurs only in $\sub_\A (\lrx_{i-1}\lrx_j)$.  
Therefore the column of $A_1$ corresponding to $\lrx_i\lrx_j$ has a zero in every row except for row $\lrx_{i-1}\lrx_j$, in which there is a one.  
Therefore, using the Laplace expansion of $\det A_1$ along the columns $\{ \lrx_i\lrx_j   : j\neq i+1 \text{ and } (i,j)\neq (0,0)\}$, we can arrive at the simplified determinant
\[
\det A_1 = (-1)^\epsilon \det B,
\]
where $\epsilon = 0$ or $1$ and $B$ is the submatrix of $A_1$ obtained by eliminating the columns $\{ \lrx_i\lrx_j   : j\neq i+1 \text{ and }(i,j)\neq (0,0)\}$ and the rows $\{ \lrx_{i-1}\lrx_j  : j\neq i+1 \text{ and }(i,j)\neq (0,0)\}$.  
What remains are the columns $\{ \lrx_i\lrx_{i+1} : 0\leq i < n\}\cup \{\lrx_0\lrx_0\}$ and the rows $\{\lrx_{i-1}\lrx_{i+1}  : 0\leq i< n\}\cup \{\lrx_{n-1}\lrx_0\}$.  

Let $\vectorv$ be a vector in the nullspace of $B$, and index the entries of $\vectorv$ by $\{ \lrx_{i}\lrx_{i+1}  : 0\leq i < n\}\cup\{\lrx_0\lrx_0\}$.  
If $i\neq n-1$, then multiplying $\vectorv$ by row $\lrx_{i-2}\lrx_i$ of $B$ yields the equation $\vectorv_{\lrx_{i-1}\lrx_i}+\vectorv_{\lrx_i\lrx_{i+1}} = 0$.  
Combining these equations yields $\vectorv_{\lrx_{n-1}\lrx_0} = \pm \vectorv_{\lrx_{n-2}\lrx_{n-1}}$.  

Multiplying $\vectorv$ by row $\lrx_{n-1}\lrx_0$ of $B$ yields the equation $\vectorv_{\lrx_{0}\lrx_0} = - \vectorv_{\lrx_0\lrx_1}$, which we know is equal to  $\vectorv_{\lrx_{n-1}\lrx_0}$.  
Multiplying $\vectorv$ by row $\lrx_{n-3}\lrx_{n-1}$ of $B$ yields the equation $0 = \vectorv_{\lrx_{n-2}\lrx_{n-1}}+\vectorv_{\lrx_{n-1}\lrx_0} + \vectorv_{\lrx_{0}\lrx_0} = \vectorv_{\lrx_{n-2}\lrx_{n-1}}+2\vectorv_{\lrx_{n-1}\lrx_0}$, which contradicts the equation $\vectorv_{\lrx_{n-1}\lrx_0} = \pm \vectorv_{\lrx_{n-2}\lrx_{n-1}}$ unless $\vectorv_{\lrx_{n-1}\lrx_0} = 0$, which in turn implies that $\vectorv = 0$.    

Therefore $B$ is non-singular, so $A_1$ is non-singular.  
Then, the quotient matrix $\tilde{A}_1$ induced by $A_1$ on $C^1(AP_\collar(\A))/\delta(C^0(AP_\collar(\A)))$ is also non-singular, because $\delta(C^0(AP_\collar(\A)))$ generates an invariant subspace for $A_1$, and hence also for $A_1^{-1}$.  
Therefore $\varprojlim_{\tilde{A}_1}\Z^{n^2-n+1}$ has rank $n^2-n+1$.  
\end{proof}

\begin{remark}
If $\A$ has an odd number of letters then, to achieve maximum rank in $H^1$, we can instead use the simpler substitution $\sub_\A'$, defined by
\[
\sub_\A'(\lrx_i) = \lrx_i\lrx_{i+1},
\]
where addition in the indices is done modulo $n$.  
\end{remark}

\subsection{A Subgroup of $H^1$}

Now let us describe a particular subgroup of $C^1(AP_\collar(\A))$ that is invariant under the action of $\shift_{i,\collar}^*(1)$.  
\begin{prop}\alabel{PROP:small-subgroup}
Let $\sub$ be a substitution on an alphabet $\A$.  
Let $\shift$ denote the map induced by $\sub$ on $AP_\collar(\A)$, and let $\shift^*$ denote the cellular map induced by $\shift$.  
Let $G$ denote the subgroup of $C^1(AP_\collar(\A))$ generated by the cochains $(* \lrv)'$, $\lrv\in \A$, where
\[
(* \lrv)' := \sum_{\lru\in\A}(\lru\lrv)'.
\]

Then $\shift^*(G) \subseteq G$.
\end{prop}
\begin{proof}
Given $\lru, \lrv\in\A$, let $n_{\lru,\lrv}$ denote the number of occurrences of the letter $\lrv$ in the word $\sub(\lru)$.  
Then
\[
\shift^*((* \lrv)') = \sum_{\lru\in\A}n_{\lru,\lrv}(* \lru)'.
\]
\end{proof}

\subsection{An Example}

The next example shows that interesting things can happen in moving from single substitution tiling systems to mixed substitution tiling systems.  
In particular, in this example the rank of $H^1(\TS_{\family,\seq})$ is seen to vary depending on the choice of $\seq$.  
\begin{example}\alabel{EX:long-example}
Let $\A = \{ a,b,c\}$ and define two substitutions $\sub_1$ and $\sub_2$ by 
\[
\sub_1 : \begin{array}{ll} %
a & \mapsto ab\\%
b & \mapsto bc\\%
c & \mapsto ca%
\end{array}, \qquad
\sub_2 : \begin{array}{ll} %
a & \mapsto bb\\%
b & \mapsto cc\\%
c & \mapsto ac%
\end{array}.
\]

Then $\family := \{ \sub_1,\sub_2\}$ is self-correcting; indeed, if $\psi = \sub_{i_1}\sub_{i_2}\sub_{i_3}\sub_{i_4}\sub_{i_5}$, where $i_1,\ldots, i_5\in \{ 1,2\}$, then the words in $\psi(\A)$ and $\psi(\A^2)$ have the same set of two-letter subwords.  
Also, $\family$ is primitive in the sense of Definition \ref{DEF:mixed-primitive}; that is, $(\TS_{\family,\seq},\R)$ is a minimal dynamical system for each sequence $\seq$.  

Let $A_1$ and $A_2$ respectively denote the matrices of the $1$-cell maps induced by these substitutions on $AP_\collar(\A)$.  
Then
\[
A_1 = \left[ \begin{array}{r@{\ \ }r@{\ \ }r@{\ \ }r@{\ \ }r@{\ \ }r@{\ \ }r@{\ \ }r@{\ \ }r}
0& 1& 0& 1& 0& 0& 0& 0& 0\\
0& 0& 0& 0& 1& 1& 0& 0& 0\\
0& 0& 0& 0& 0& 1& 1& 0& 0\\
0& 1& 0& 0& 0& 0& 1& 0& 0\\
0& 0& 0& 0& 0& 1& 0& 1& 0\\
0& 0& 0& 0& 0& 0& 1& 0& 1\\
1& 1& 0& 0& 0& 0& 0& 0& 0\\
0& 1& 0& 0& 0& 1& 0& 0& 0\\
0& 0& 1& 0& 0& 0& 1& 0& 0
\end{array}\right],\ 
A_2 = \left[ \begin{array}{r@{\ \ }r@{\ \ }r@{\ \ }r@{\ \ }r@{\ \ }r@{\ \ }r@{\ \ }r@{\ \ }r}
0& 0& 0& 0& 2& 0& 0& 0& 0\\
0& 0& 0& 0& 0& 1& 0& 0& 1\\
0& 0& 1& 1& 0& 0& 0& 0& 0\\
0& 0& 0& 0& 1& 0& 0& 1& 0\\
0& 0& 0& 0& 0& 0& 0& 0& 2\\
0& 0& 1& 0& 0& 0& 1& 0& 0\\
0& 0& 0& 0& 1& 0& 0& 1& 0\\
0& 0& 0& 0& 0& 0& 0& 0& 2\\
0& 0& 1& 0& 0& 0& 1& 0& 0
\end{array}\right].
\]

After reducing modulo the images of coboundaries, these become
\[
\tilde{A}_1 = \left[ \begin{array}{rrrrrrr}
0& 1& 0& 0& 0& 0& 0\\
0& 0& 1& 1& 1& 0& 0\\
0& 0& 0& 1& 0& 1& 0\\
0& 0& 1& 0& 0& 0& 1\\
1& 0& 0& 1& 1& 0& 0\\
0& 0& -1& 1& 1& 0& 0\\
0& 0& 0& 0& 1& 0& 0
\end{array}\right],\ 
\tilde{A}_2 = \left[ \begin{array}{rrrrrrr}
0& 0& 2& 0& 0& 0& 0\\
0& 0& 1& 1& 0& 1& 1\\
0& 0& 0& 0& 0& 0& 2\\
0& -1& 0& 1& 1& 0& 1\\
0& 1& 1& 0& 0& 1& 0\\
0& 1& 0& -1& 0& 0& 1\\
0& 0& 0& 0& 1& 0& 0
\end{array}\right].
\]

$\tilde{A}_1$ has rank $7$, but $\tilde{A}_2$ has rank $5$, with a generalized $0$-eigenspace of dimension $4$.  

Consider the following basis $\basisb$ for $\Q^7$.  
\[
\begin{array}{ccccccc}
\vectorv_1 &
\vectorv_2 &
\vectorv_3 &
\vectorv_4 &
\vectorv_5 &
\vectorv_6 &
\vectorv_7 \\
\left[ \begin{array}{r}
 1 \\
 2 \\
 1 \\
 1 \\
 2 \\
 1 \\
 1 %
\end{array}
\right] &
\left[ \begin{array}{r}
 1 \\
 0 \\
-1 \\
-1 \\
 1 \\
 0 \\
 0 %
\end{array}
\right] &
\left[ \begin{array}{r}
-1 \\
-2 \\
-1 \\
-1 \\
 0 \\
 1 \\
 1 %
\end{array}
\right] &
\left[ \begin{array}{r}
 0 \\
-1 \\
 0 \\
-1 \\
 0 \\
 1 \\
 0 %
\end{array}
\right] &
\left[ \begin{array}{r}
 1 \\
 0 \\
 0 \\
 0 \\
 0 \\
 0 \\
 0 %
\end{array}
\right] &
\left[ \begin{array}{r}
 0 \\
 0 \\
 0 \\
 1 \\
 0 \\
 0 \\
 0 %
\end{array}
\right] &
\left[ \begin{array}{r}
 0 \\
 0 \\
 1 \\
 0 \\
 0 \\
-1 \\
 0 %
\end{array}
\right] 
\end{array}.
\]

$\vectorv_1$ is a $2$-eigenvector for both $\tilde{A}_1$ and $\tilde{A}_2$.  
$\vectorv_2$ and $\vectorv_3$ are integer vectors that span the direct sum of the eigenspaces of $\tilde{A}_2$ with the complex eigenvectors $\frac{-1\pm\sqrt{-7}}{2}$.  
Together, $\vectorv_1,\vectorv_2,$ and $\vectorv_3$ are a basis for the subspace spanned by the group $G$ from Proposition \ref{PROP:small-subgroup}.  
Therefore $\spa_\Q \{\vectorv_1,\vectorv_2,\vectorv_3\}$ is invariant for both $\tilde{A}_1$ and $\tilde{A}_2$, so $H^1(\TS_{\family,\seq})$ has rank at least $3$ regardless of $\seq$.  

$\vectorv_4$ and $\vectorv_5$ are $0$-eigenvectors of $\tilde{A}_2$, and $\vectorv_6$ and $\vectorv_7$ are generalized $0$-eigenvectors: $\tilde{A}_2\vectorv_6 = \vectorv_4$ and $\tilde{A}_2\vectorv_7 = \vectorv_5$.  

With respect to this basis, $\tilde{A}_1$ equals
\[
\left[ \begin{array}{rrrrrrr}
2& -\frac{1}{2}& \frac{1}{4}& -\frac{3}{8}& \frac{1}{8}& \frac{1}{2}& -\frac{1}{8}\\
0& \frac{1}{2}& \frac{3}{4}& \frac{3}{8}& -\frac{1}{8}& -\frac{1}{2}& \frac{1}{8}\\
0& -1& \frac{1}{2}& -\frac{1}{4}& \frac{3}{4}& 0& \frac{1}{4}\\
0& 0& 0& -\frac{1}{2}& \frac{1}{2}& 1& -\frac{3}{2}\\
0& 0& 0& 0& -1& -1& 0\\
0& 0& 0& 0& 1& 0& 0\\
0& 0& 0& \frac{1}{2}& \frac{1}{2}& 0& -\frac{1}{2}
\end{array}\right],
\]
and its square and cube, respectively, are
\[
\left[\begin{array}{r@{\ }r@{\ }r@{\ }r@{\ }r@{\ }r@{\ }r}
4& -\frac{3}{2}& \frac{1}{4}& -\frac{7}{8}& \frac{5}{8}& \frac{3}{4}& \frac{3}{8}\\
0& -\frac{1}{2}& \frac{3}{4}& -\frac{1}{8}& \frac{3}{8}& \frac{1}{4}& -\frac{3}{8}\\
0& -1& -\frac{1}{2}& -\frac{1}{4}& -\frac{1}{4}& -\frac{1}{2}& \frac{1}{4}\\
0& 0& 0& -\frac{1}{2}& -\frac{1}{2}& -1& \frac{3}{2}\\
0& 0& 0& 0& 0& 1& 0\\
0& 0& 0& 0& -1& -1& 0\\
0& 0& 0& -\frac{1}{2}& -\frac{1}{2}& 0& -\frac{1}{2}
\end{array}\right] \ \text{and}\ 
\left[\begin{array}{r@{\ }r@{\ }r@{\ }r@{\ }r@{\ }r@{\ \ }r}
8& -3& 0& -\frac{3}{2}& \frac{3}{4}& \frac{5}{4}& \frac{1}{2}\\
0& -1& 0& -\frac{1}{2}& \frac{1}{4}& -\frac{1}{4}& \frac{1}{2}\\
0& 0& -1& 0& -\frac{1}{2}& \frac{1}{2}& 0\\
0& 0& 0& 1& 0& 0& 0\\
0& 0& 0& 0& 1& 0& 0\\
0& 0& 0& 0& 0& 1& 0\\
0& 0& 0& 0& 0& 0& 1
\end{array}\right].
\]

Notice in particular that $\tilde{A}_1^3$ is upper triangular with respect to $\basisb$, while $\tilde{A}_1$ and $\tilde{A}_1^2$ are not.  
Therefore the same can be said of the inverses of these matrices: with respect to $\basisb$, $\tilde{A}_1^{-3}$ is upper triangular, and $\tilde{A}_1^{-2}$ and $\tilde{A}_1^{-1}$ are not.  

This means that $\tilde{A}_1^{-3}\vectorv_4$ and $\tilde{A}_1^{-3}\vectorv_5$ lie in the span of $\{ \vectorv_1,\vectorv_2,\vectorv_3,\vectorv_4,\vectorv_5\}$, which is contained in the range of $\tilde{A}_2$.  
Therefore the range of $\tilde{A}_1^{3i}\tilde{A}_2$ contains two non-zero vectors in the kernel of $\tilde{A}_2$, so $\tilde{A}_2\tilde{A}_1^{3i}\tilde{A}_2$ has a kernel of dimension four.  

On the other hand, $\tilde{A}_1^{-1}\vectorv_4$, $\tilde{A}_1^{-1}\vectorv_5$, $\tilde{A}_1^{-2}\vectorv_4$, and $\tilde{A}_1^{-2}\vectorv_5$ all have non-zero $\vectorv_6$- and $\vectorv_7$-coefficients in their $\basisb$-expansions.  
Since $\vectorv_6$ and $\vectorv_7$ are not in the range of $\tilde{A}_2$, this means that $\tilde{A}_2\tilde{A}_1^{3i+1}\tilde{A}_2$ and $\tilde{A}_2\tilde{A}_1^{3i+2}\tilde{A}_2$ have kernels of dimension two.  

Therefore rank$(H^1(\TS_{\family,\seq})) = $
\begin{enumerate}[(i)]
\item  $7$ if $\seq$ contains only finitely many $2$s.  
\item  $5$ if $\seq$ contains only finitely many subsequences of the form $21^{3i}2$.  
\item  $3$ if $\seq$ contains infinitely many subsequences of the form $21^{3i}2$.  
\end{enumerate}
\end{example}
